%% file: monoidal.galois.tex
\documentclass[oneside,english,a4paper]{amsart}
\usepackage[T1]{fontenc}
\usepackage[latin9]{inputenc}
\usepackage{verbatim}
\usepackage{mathrsfs}
\usepackage{amsthm}
\usepackage{amstext}
\usepackage{amssymb}

\makeatletter
\numberwithin{equation}{section}
\numberwithin{figure}{section}
\numberwithin{table}{section}
  \theoremstyle{plain}
  \newtheorem*{thm*}{\protect\theoremname}
\theoremstyle{plain}
\newtheorem{thm}{\protect\theoremname}[section]
  \theoremstyle{definition}
  \newtheorem{defn}[thm]{\protect\definitionname}
  \theoremstyle{remark}
  \newtheorem{rem}[thm]{\protect\remarkname}
  \theoremstyle{plain}
  \newtheorem{prop}[thm]{\protect\propositionname}
  \theoremstyle{definition}
  \newtheorem{example}[thm]{\protect\examplename}
  \theoremstyle{plain}
  \newtheorem{lem}[thm]{\protect\lemmaname}
  \theoremstyle{plain}
  \newtheorem{cor}[thm]{\protect\corollaryname}

\input{packages_and_functions.for_preamble}
\usepackage{enumitem}
\setenumerate[1]{label={\arabic*)}} 
\usepackage[a4paper]{geometry}
\renewcommand{\Lex}{\textup{Mon}}

\theoremstyle{plain}
\newtheorem{thma}{Theorem}
\newtheorem{thmb}{Theorem}
\newtheorem{thmc}{Theorem}

\author[Tonini]{Fabio Tonini}

\address[Tonini]{Freie Universit\"at Berlin\\
    FB Mathematik und Informatik\\
    Arnimallee 3\\ Zimmer 112A\\
    14195 Berlin\\ Deutschland}
\email{tonini@zedat.fu-berlin.de}

\makeatother

\usepackage{babel}
  \providecommand{\corollaryname}{Corollary}
  \providecommand{\definitionname}{Definition}
  \providecommand{\examplename}{Example}
  \providecommand{\lemmaname}{Lemma}
  \providecommand{\propositionname}{Proposition}
  \providecommand{\remarkname}{Remark}
  \providecommand{\theoremname}{Theorem}
\providecommand{\theoremname}{Theorem}

\begin{document}

\title{Ramified Galois covers via monoidal functors}

\maketitle
\input{packages_and_functions.tex}

\begin{abstract}
We interpret Galois covers in terms of particular monoidal functors,
extending the correspondence between torsors and fiber functors. As
applications we characterize tame $G$-covers between normal varieties
for finite and étale group schemes and we prove that, if $G$ is a
finite, flat and finitely presented nonabelian and linearly reductive
group scheme over a ring, then the moduli stack of $G$-covers is
reducible.
\end{abstract}

\section*{Introduction}

Let $R$ be a base commutative ring and $G$ be a flat, finite and
finitely presented group scheme over $R$. In \cite{Tonini2011} I
introduced the notion of a ramified Galois cover with group $G$,
briefly a $G$-cover, and the stack $\GCov$ of such objects (see
\ref{def: G cover} for details). This stack is algebraic and of finite
type over $R$ and contains $\Bi_{R}G$, the stack of $G$-torsors,
as an open substack. If $G$ is diagonalizable, its nice representation
theory makes it possible to study $G$-covers in terms of simplified
data (collections of invertible sheaves and morphisms between them)
and to investigate the geometry of the moduli $\GCov$ (see \cite{Tonini2011}). 

The general case is much harder, even when $G$ is a constant group
over an algebraically closed field of characteristic zero: a direct
approach as in the diagonalizable case fails because of the complexity
of the representation theory of $G$. Thus in order to handle general
$G$-covers one needs a different perspective and Tannaka's duality
comes into play. The $G$-torsors are very special $G$-covers and
the solution of Tannaka's reconstruction problem asserts that they
can be described in terms of particular strong monoidal functors with
domain $\Loc^{G}R$, the category of $G$-comodules over $R$ which
are projective and finitely generated as $R$-modules. If $\stX$
is an algebraic stack, denote by $\Loc\stX$ (resp. $\QCoh\stX$)
the category of locally free of finite rank (resp. quasi-coherent)
sheaves on $\stX$, so that $\Loc\Bi_{R}G\simeq\Loc^{G}R$. When $\stX=\Spec A$
we simply write $\Loc A$ and $\QCoh A$. The result about $G$-torsors
can be stated as follows.
\begin{thm*}
(\cite[Theorem 3.2]{Deligne1982}, \cite[Theorem 1.3.2]{Schappi2011})
Let $\SMon_{R}^{G}$ be the stack over $R$ whose fiber over an $R$-scheme
$T$ is the category of $R$-linear, exact (on short exact sequences)
and strong monoidal functors $\Loc^{G}R\arr\Loc T$. Then the functor
  \[   \begin{tikzpicture}[xscale=3.2,yscale=-0.7]     \node (A0_0) at (0, 0) {$\Bi_R G $};     \node (A0_1) at (1, 0) {$\SMon^G_R$};     \node (A1_0) at (0, 1) {$( T\arrdi s \Bi_R G)$};     \node (A1_1) at (1, 1) {$s^*_{|\Loc^GR}$};     \path (A0_0) edge [->]node [auto] {$\scriptstyle{\Delta}$} (A0_1);     \path (A1_0) edge [|->,gray]node [auto] {$\scriptstyle{}$} (A1_1);   \end{tikzpicture}   \] is
an equivalence of stacks.
\end{thm*}
Since a $G$-cover is a ``weak'' version of a $G$-torsor it is
natural to look at a ``weak'' version of a strong monoidal functor,
that is, as the words suggest, a (lax) monoidal functor. This idea
has motivated the study in \cite{Tonini2014} of more general monoidal
(and non) functors and this paper is an application of it. We introduce
the stack $\Lex_{R}^{G}$ ($\Lex_{R,\text{reg}}^{G}$) over $R$ whose
fiber over an $R$-scheme $T$ is the groupoid of $R$-linear, exact
monoidal functors $\Gamma\colon\Loc^{G}R\arr\Loc T$ (such that $\rk\Gamma_{V}=\rk V$
(pointwise) for all $V\in\Loc^{G}R$). We also denote by $\LAlg_{R}^{G}$
the stack over $R$ whose fiber over an $R$-scheme $T$ is the groupoid
of locally free sheaves of algebras on $T$ with an action of $G$,
or, alternatively, the stack of covers with an action of $G$. The
stack $\LAlg_{R}^{G}$ is algebraic and locally of finite presentation
over $R$ and $\GCov$ is an open substack of $\LAlg_{R}^{G}$ (see
\ref{prop:LAlg of finite presentation}). 

Recall that $G$ is linearly reductive over $R$ if the functor of
invariants $(-)^{G}\colon\QCoh\Bi_{R}G\arr\QCoh R$ is exact. We say
that $G$ has a \emph{good} representation theory over $R$ if it
is linearly reductive and there exists a finite collection $I_{G}$
of sheaves in $\Loc^{G}R$ such that for all geometric points (one
is enough if $\Spec R$ is connected) $\Spec k\arr\Spec R$ the map
$(-\otimes_{R}k)\colon I_{G}\arr\Loc^{G}k$ is a bijection onto a
collection of representatives of the irreducible representations of
$G\times_{R}k$. Examples of groups with a good representation theory
are diagonalizable groups and linearly reductive groups over algebraically
closed fields. In general we show that any linearly reductive group
$G$ over $R$ has fppf locally (étale locally if $G/R$ is étale)
a good representation theory (see \ref{prop:G has locally a good representation theory}).

\begin{thma}\label{A} The map of stacks 
\[
\widetilde{\Delta}\colon\GCov\arr\Lex_{R}^{G}\comma(X\arrdi fT)\longmapsto(f_{*}\odi X\otimes-)^{G}
\]
is an open immersion, it extends the equivalence $\Delta\colon\Bi_{R}G\arr\SMon_{R}^{G}$
and takes values in $\Lex_{R,\text{reg}}^{G}$. 

If $G$ is linearly reductive over $R$ then $\widetilde{\Delta}$
extends to an equivalence $\widetilde{\Delta}\colon\LAlg_{R}^{G}\arr\Lex_{R}^{G}$,
namely $\widetilde{\Delta}(\alA)=(\alA\otimes-)^{G}$, the stack $\GCov$
is an open and closed substack of $\LAlg_{R}^{G}$ and, if $G$ has
a good representation theory, then $\widetilde{\Delta}(\GCov)=\Lex_{R,\text{reg}}^{G}$.
\end{thma}

The equality $\widetilde{\Delta}(\GCov)=\Lex_{R,\text{reg}}^{G}$
is not true in general, even when $G$ is linearly reductive (see
\ref{rem: counterexample for gcovers and monoidal}).

We are going to show two applications of the above point of view.
The first one is about the geometry of $\GCov$ (see also \ref{thm:GCov reducible when G not abelian}).

\begin{thmb}\label{B} If $G$ is a finite, flat and finitely presented
nonabelian linearly reductive group scheme over $R$ then the stack
$\GCov$ is reducible.

\end{thmb}

When $G$ is a diagonalizable group the same result holds except for
a few cases when $G$ has low rank (see \cite[Corollary 4.17]{Tonini2011}).
Thus the bad behaviour of the moduli $\GCov$ is still present in
the nonabelian setting. Note that the proof of Theorem \ref{B} does
not use and cannot be adapted to show the reducibility of $\GCov$
when $G$ is a diagonalizable group. Moreover it requires the study
of more general monoidal functors than the ones present in $\Lex_{R,\text{reg}}^{G}$.
Theorem \ref{B} already appeared in my Ph.D. thesis \cite{Tonini2013},
but the proof we present here is slightly different and relies on
the following fact: if $H$ is an open and closed subgroup scheme
of $G$ the functor
\[
\ind_{H}^{G}\colon\LAlg_{R}^{H}\arr\LAlg_{R}^{G}\comma\alA\longmapsto(\alA\otimes R[G])^{H}
\]
is well defined, quasi-affine and étale (see \ref{thm:ind is etale}).

The second application is a characterization of $G$-covers of regular
in codimension $1$ schemes. Let us introduce some notation and definitions
in order to explain the result. Let $f\colon X\arr T$ be a cover
with an action of $G$ on $X$. We denote by $\tr_{f}\colon f_{*}\odi X\arr\odi T$
the trace map, by $\widetilde{\tr}_{f}\colon f_{*}\odi X\arr\duale{(f_{*}\odi X)}$
the map $x\longmapsto\tr_{f}(x\cdot-)$ and by $s_{f}\in(\det f_{*}\odi X)^{-2}$
the discriminant section, that is the section obtained by $\det\widetilde{\tr}_{f}$.
If $f$ is a $G$-cover with associated monoidal functor $\Omega^{f}=(f_{*}\odi X\otimes-)^{G}\colon\Loc^{G}R\arr\Loc T$
and $V\in\Loc^{G}R$ consider
\[
\Omega_{V}^{f}\otimes\Omega_{\duale V}^{f}\arr\Omega_{V\otimes\duale V}^{f}\arr\Omega_{R}^{f}=(f_{*}\odi X)^{G}=\odi T
\]
where the first map is given by monoidality, while the second is induced
by the evaluation $V\otimes\duale V\arr R$. The morphism above yields
a map $\xi_{f,V}\colon\Omega_{\duale V}^{f}\arr\duale{(\Omega_{V}^{f})}$
of locally free sheaves whose rank coincides with $\rk V$ by Theorem
\ref{A}. Applying the determinant we obtain a section $s_{f,V}\in(\det\Omega_{V}^{f}\otimes\det\Omega_{\duale V}^{f})^{-1}$.
If $q\in T$ is a point and $V\in\Loc^{G}T$ we denote by $\rk_{q}V$
the rank of $V\otimes\odi{T,q}$ and by $\rk_{q}G$ the rank of $G$
over $q$, that is $\rk_{q}\odi T[G]$. The result we will prove is
the following.

\begin{thmc}\label{C} Let $G$ be a finite and étale group scheme
over $R$. Let also $Y$ be an integral and Noetherian $R$-scheme
with $\dim Y\geq1$, and $f\colon X\arr Y$ be a cover with an action
of $G$ on $X$ over $Y$ and such that $X/G=Y$. Let also $q\in Y$
be a codimension $1$ and regular point. Then the following are equivalent:
\begin{enumerate}
\item all points of $X$ over $q$ are regular, tame (the ramification index
is coprime with $\car k(q)$) and have separable residue fields. 
\item we have $v_{q}(s_{f})<\rk f$, where $v_{q}$ denotes the valuation
in $q$;
\item there exist an étale neighborhood $U\arr Y$ with a point $q'$ mapping
to $q$ and with $G\times U$ constant, subgroups $T\vartriangleleft H<G\times U$
with $H/T$ cyclic of order coprime with $\car k(q)$ and $\Spec\alB\in(H/T)\textup{-Cov}(U)$
such that $X\times_{Y}U=\Spec(\ind_{H}^{G}\alB)$, $\alB_{q'}$ is
a regular local ring, $H$ is the geometric stabilizer of a codimension
$1$ point of $X$ over $q$, $T$ is the geometric stabilizer of
a generic point of $X$ and $\Spec\alB$ is generically an $(H/T)$-torsor.
\end{enumerate}
If one of the above conditions is satisfied we have that: $f$ is
generically a $G$-torsor if and only if $\rk f=\rk G$ and in this
case the geometric stabilizers of the codimension $1$ points of $X$
over $q$ are linearly reductive and cyclic and there exists an open
subset $V\subseteq Y$ containing $q$ and such that $f_{|f^{-1}(V)}\colon f^{-1}(V)\arr V$
is a $G$-cover; if $G$ is constant, $G\arr\Aut X$ is injective
and the generic fiber of $f\colon X\arr Y$ is connected then $\rk f=\rk G$.

If $G$ is linearly reductive and $\rk f=\rk G$ then the above conditions
are equivalent to
\begin{enumerate}
\item [4)] $f\in\GCov$ and for all $V\in\Rep^{G}R$ (resp. $V\in I_{G}$
if $G$ is good) we have $v_{q}(s_{f,V})\leq\rk_{q}(V/V^{G})$;
\item [5)] $f\in\GCov$ and for all $V\in\Rep^{G}R$ (resp. $V\in I_{G}$
if $G$ is good) we have that $\Coker(\xi_{f,V})\otimes\odi{Y,q}$
is defined over $k(q)$, that is $m_{q}(\Coker(\xi_{f,V})\otimes\odi{Y,q})=0$
where $m_{q}$ denotes the maximal ideal of $\odi{Y,q}$.
\end{enumerate}
In this case $f\in\stZ_{G}(Y)$, where $\stZ_{G}$ denotes the schematic
closure of $\Bi G$ inside $\GCov$ (see \ref{def:schematic closure}).

\end{thmc}

A variant of this result already appeared in my Ph.D. thesis \cite{Tonini2013}
but under stronger hypotheses on the geometric stabilizers in codimension
$1$ (see \cite[Theorem 4.4.7]{Tonini2013}). The proof we present
here is different and relies on \cite{Tonini2015}, where a non-equivariant
analogue of the above theorem is proved.

We now briefly describe the subdivision of the paper. In the first
section we prove Theorem \ref{A}, while in the second we study the
property of induction from an open and closed subgroup. The third
section is dedicated to the proof of Theorem \ref{B} and the fourth
section to the proof of Theorem \ref{C}.

\section*{Notation}

In all the paper we fix a base ring $R$, so that all rings, schemes
and stacks will be defined over $R$.

Consider a scheme $T$ and a finite, flat and finitely presented group
scheme $G$ over $R$. We denote by $\Bi_{R}G$ (or simply $\Bi G)$
the stack over $R$ of $G$-torsors, by $\Loc T$ (resp. $\QCoh T$)
the category of sheaves of $\odi T$-modules that are locally free
of finite rank (resp. quasi-coherent), by $\Loc^{G}T$ (resp. $\QCoh^{G}T$)
the category of sheaves of $\odi T$modules that are locally free
of finite rank (resp. quasi-coherent) together with an action of $G$,
and by $\QAlg^{G}T$ the category of quasi-coherent sheaves of algebras
$\alA$ on $T$ together with an action of $G$. When $T=\Spec A$
we will often replace $T$ by $A$ and write, for instance, $\Loc^{G}A$
instead of $\Loc^{G}(\Spec A)$.

If $\catC$, $\catD$ are $R$-linear monoidal categories with unities
$I$, $J$ and $\Gamma\colon\catC\arr\catD$ is an $R$-linear functor,
a \emph{monoidal structure} on $\Gamma$ consists of a natural transformation
$\iota_{V,W}\colon\Gamma_{V}\otimes\Gamma_{W}\arr\Gamma_{V\otimes W}$
for $V,W\in\catC$ and a morphism $1\colon J\arr\Gamma_{I}$ satisfying
certain compatibility conditions. A monoidal structure in which those
maps are isomorphisms is called \emph{strong.} We refer to \cite[Definition 2.18]{Tonini2014}
for the precise definition.

Given $\shF\in\QCoh^{G}T$ we set $\Omega^{\shF}=(\shF\otimes-)^{G}\colon\Loc^{G}R\arr\QCoh T$,
which is an $R$-linear functor. If $\shF\in\QAlg^{G}T$ then $\Omega^{\shF}$
has a monoidal structure induced by the multiplication and the unity
of $\shF$ (see \cite[Proposition 2.22 and Section 4]{Tonini2014}).

A map $f\colon X\arr T$ of schemes is called a \emph{cover} if it
is affine and $f_{*}\odi{\stX}$ is locally free of finite rank or,
alternatively, if it is finite, flat and finitely presented. Affine
maps into a scheme $T$ will be often thought of as quasi-coherent
sheaves of algebras on $T$, so that covers correspond to locally
free sheaves of algebras of finite rank.

A \emph{geometric point} of a scheme $T$ is a map $\Spec k\arr T$,
where $k$ is an algebraically closed field.

\section*{Acknowledgement}

I would like to thank Angelo Vistoli and Matthieu Romagny for the
useful conversations I had with them and all the suggestions they
gave me.

\section{Galois covers via monoidal functors}

The aim of this section is to prove Theorem \ref{A}. We fix a base
ring $R$ and a finite, flat and finitely presented group scheme $G$
over $R$.

Taking into account \cite[Remark 4.3 and Theorem 4.6]{Tonini2014}
we have the following result.
\begin{thm}
\label{thm:sheaves and linear functors} The functor $\Omega^{*}$
yields an equivalence between $\QCoh^{G}T$ ($\QAlg^{G}T$) and the
category of $R$-linear (monoidal) functors $\Loc^{G}R\arr\QCoh T$
which are left exact on short exact sequences.\end{thm}
\begin{defn}
\label{def: G cover} A $G$-\emph{cover }of an $R$-scheme $T$ is
a cover $f\colon X\arr T$ together with an action of $G$ on $X$
such that $f$ is invariant and $f_{*}\odi X$ and $R[G]\otimes\odi T$
are fppf locally isomorphic as $G$-comodules (not as rings). 

We denote by $\GCov$ the stack over $R$ of $G$-covers. The stack
$\GCov$ has been introduced in \cite{Tonini2011}, it is algebraic
and of finite type over $R$ and contains $\Bi_{R}G$ as an open substack.
\end{defn}
The following remark (see \cite[Part 1, 3.4]{Jantzen2003} for a proof)
will be often used in the next pages.
\begin{rem}
\label{rem: crucial isomorphism} If $M\in\QCoh^{G}R$ and $\varepsilon\colon R[G]\arr R$
is the counit then the evaluation in $\varepsilon$ yields an $R$-linear
isomorphism
\[
\Hom^{G}(\duale{R[G]},M)\simeq M
\]
or, equivalently, the composition $(R[G]\otimes M)^{G}\arr R[G]\otimes M\arrdi{\varepsilon\otimes\id_{M}}M$
is an $R$-linear isomorphism.\end{rem}
\begin{defn}
Given an $R$-scheme $T$ we denote by $\LAlg^{G}T$ the groupoid
of locally free sheaves of algebras over $T$ with an action of $G$
and by $\LAlg_{R}^{G}$ the stack over $R$ they form. Given $n\in\N$
we also denote by $\LAlg_{n}^{G}T$ (resp. $\LAlg_{R,n}^{G}$) the
subcategory of $\LAlg^{G}T$ (resp. substack of $\LAlg_{R}^{G}$)
of sheaves of rank $n$.\end{defn}
\begin{prop}
\label{prop:LAlg of finite presentation} We have that $\LAlg_{R}^{G}=\sqcup_{n\in\N}\LAlg_{R,n}^{G}$
and that $\LAlg_{R,n}^{G}$ is an algebraic stack of finite presentation
over $R$ for all $n\in\N$. Moreover the map
\[
\GCov\arr\LAlg_{R}^{G}\comma(f\colon X\arr Y)\longmapsto f_{*}\odi X
\]
is an open immersion.\end{prop}
\begin{proof}
The first claim follows from the fact that the rank function for a
locally free sheaf is locally constant. For the second one, consider
the forgetful functor $\LAlg_{R,n}^{G}\arr\Bi\GL_{n}$ and call $X$
the fiber product along the universal torsor $\Spec R\arr\Bi\GL_{n}$.
For simplicity we can assume that $R[G]$ is free as an $R$-module.
The stack $X$ is actually a sheaf $X\colon(\Sch/R)^{\op}\arr\sets$
and it maps a scheme $T$ to the set of all possible ring structures
together with an action of $G$ on $\odi T^{n}$. Since a ring structure
is given by maps $\odi T^{n}\otimes\odi T^{n}\arr\odi T^{n}$ (the
multiplication) and $\odi T\arr\odi T^{n}$(the unity), while a $R[G]$-comodule
structure by a map $\odi T^{n}\arr\odi T^{n}\otimes R[G]$ (the comodule
structure), we can embed $X$ into an affine space $\A^{N}$. The
compatibility conditions among the previous maps allow us to conclude
that $X$ is the zero locus in $\A^{N}$of finitely many polynomials,
as required.

We now deal with the last claim. Clearly the map in the statement
is fully faithful. We have to prove that if $\alA\in\LAlg^{G}B$,
where $B$ is a ring, then the locus in $\Spec B$ where $\alA$ is
fppf locally the regular representation is open. Concretely, if $\xi\colon\Spec k\arr\Spec B$
is a geometric point and $\alA\otimes k\in\GCov(k)$ we will prove
that there exists a flat and finitely presented map $\Spec B'\arr\Spec B$
through which $\xi$ factors and such that $\alA\otimes B'\simeq B'[G]$.
Denote by $p\in\Spec B$ the image of $\xi$. Both the stack $\GCov$
and $\LAlg_{R}^{G}$ are locally of finite type over $R$ and therefore
also the map $\GCov\arr\LAlg_{R}^{G}$ is so, which in particular
implies that $\alA\otimes\overline{k(p)}\in\GCov(\overline{k(p)})$.
Thus we can assume $k=\overline{k(p)}$. Since $k$ is algebraically
closed we have that $\alA\otimes k$ is the regular representation
and thus we have a $G$-equivariant isomorphism $\overline{\omega}\colon\duale{k[G]}\arr\duale{(\alA\otimes k)}$.
By \ref{rem: crucial isomorphism} the map $\overline{\omega}$ is
completely determined by a $\overline{\phi}\in\duale{\alA}\otimes k$.
There exists a finite field extension $L/k(p)$ such that $\overline{\phi}$
comes from some element in $\duale{\alA}\otimes L$ and it is a general
fact that we can find an fppf neighborhood $\Spec B'$ of $p$ in
$\Spec B$ with a point $p'\in\Spec B'$ over $p$ such that $k(p')=L$.
Up to shrinking $\Spec B'$ around $p'$ we can assume we have $\phi\in\duale{\alA}$
inducing $\overline{\phi}$. The element $\phi$ defines a $G$-equivariant
map $\omega\colon\duale{B[G]}\arr\duale{\alA}$ of locally free sheaves
on $A$ inducing $\overline{\omega}$. Since $\overline{\omega}$
is an isomorphism it follows that $\omega$ is an isomorphism in a
Zariski open neighborhood of $p$ as required.
\end{proof}

\begin{proof}
[Proof of Theorem \ref{A}, first sentence.] Let $A$ be an $R$-algebra.
By \ref{rem: crucial isomorphism} we have 
\[
\Omega_{V}^{A[G]}=(A[G]\otimes(V\otimes A))^{G}\simeq V\otimes A\text{ for }V\in\Loc^{G}R
\]
More precisely $\Omega^{A[G]}$ is isomorphic to the forgetful functor
$(-\otimes_{R}A)\colon\Loc^{G}R\arr\Loc A$ as monoidal functor. In
particular if $\alA\in\QAlg^{G}A$ is fppf locally isomorphic to $A[G]$
(without ring structure) then the functor $\Omega^{\alA}=(\alA\otimes-)^{G}\colon\Loc^{G}R\arr\QCoh A$
is fppf locally $R$-linearly isomorphic to the forgetful functor
$(-\otimes_{R}A)\colon\Loc^{G}R\arr\Loc A$ (without monoidal structure).
This easily implies that $\widetilde{\Delta}$ is well defined and
takes values in $\Lex_{R,\text{reg}}^{G}$. It is fully faithful thanks
to \ref{thm:sheaves and linear functors}. It extends the functor
$\Delta$ because if $f\colon X\arr\Spec A$ is a $G$-torsor corresponding
to $s\colon\Spec A\arr\Bi_{R}G$ then $s_{*}\odi A\simeq f_{*}\odi X$
as sheaves of algebras on $\Bi_{R}G$ and
\[
(s_{*}\odi A\otimes_{R}V)^{G}\simeq\Hom_{\Bi_{R}G}(\duale V,s_{*}\odi A)\simeq\Hom_{A}(s^{*}\duale V,A)\simeq s^{*}V\text{ for }V\in\Loc(\Bi_{R}G)=\Loc^{G}R
\]
We now prove that it is an open immersion. Let $\Gamma\in\Lex_{R}^{G}(A)$.
By \ref{thm:sheaves and linear functors} there exists $\alA\in\QAlg^{G}A$
such that $\Gamma\simeq\Omega^{\alA}$. By definition of $\Lex_{R}^{G}$
and taking into account \ref{rem: crucial isomorphism} we also have
that $\Omega_{R[G]}^{\alA}=(\alA\otimes R[G])^{G}\simeq\alA$ is a
locally free sheaf on $A$, that is $\alA\in\LAlg^{G}A$. The result
then follows because, by \ref{prop:LAlg of finite presentation},
the locus in $\Spec A$ where $\alA$ is fppf locally the regular
representation is open.\end{proof}
\begin{defn}
The group scheme $G$ is called \emph{linearly reductive} over $R$
if the functor of invariants
\[
(-)^{G}\colon\Mod^{G}R\arr\Mod R
\]
is exact.
\end{defn}
From now until the end of the section we will assume that $G$ is
linearly reductive over $R$. Remember that this condition is stable
under base change, is local in the fppf topology and that $G$ is
fppf locally well-split, which means isomorphic to a semidirect product
of a diagonalizable group scheme and a constant group whose order
is invertible in the base ring (see \cite[Proposition 2.6, Theorem 2.19]{Abramovich2007}).
We summarize some properties of linearly reductive groups we are going
to use.
\begin{prop}
\label{prop:properties linearly reductive groups} Let $T$ be an
$R$-scheme and $A$ be an $R$-algebra. Then
\begin{enumerate}
\item If $\shF\in\QCoh^{G}T$ and $\shH\in\QCoh T$ then the natural map
\[
\shF^{G}\otimes\shH\arr(\shF\otimes\shH)^{G}
\]
where the action of $G$ on $\shH$ is trivial, is an isomorphism.
In particular taking invariants $(-)^{G}\colon\QCoh^{G}T\arr\QCoh T$
commutes with arbitrary base changes.
\item If $\shF\in\QCoh^{G}T$ is locally free of finite rank then the map
$\shF^{\shG}\arr\shF$ locally splits. In particular $\shF^{G}$ is
locally free of finite rank.
\item Every short exact sequence in $\QCoh^{G}A$ of sheaves in $\Loc^{G}A$
splits. In particular any $R$-linear functor from $\Loc^{G}R$ to
an $R$-linear category is automatically exact.
\item If $R$ is a field any finite-dimensional representation of $G$ is
a direct sum of irreducible representations. 
\end{enumerate}
\end{prop}
\begin{proof}
We can assume $T$ affine, say $T=\Spec A$ and replace $\shF,\shH$
with modules $F,H$ respectively. Point $1)$ follows because the
map in the statement is an isomorphism when $H$ is free and, in general,
using a presentation of $H$ and using the exactness of $(-)^{G}$.
Point $1)$ implies that $F^{G}\arr F$ is universally injective,
so that point $2)$ follows from \cite[Theorem 7.14]{Matsumura1989}
after reducing to a Noetherian base (for instance assuming that $G$
is well-split and, thus, defined over $\Z$). For $3)$, if $0\arr V\arr W\arr Z\arr0$
is an exact sequence of sheaves in $\Loc^{G}A$, then $\Hom(W,V)\arr\Hom(V,V)$
is surjective and, taking invariants, we can find an equivariant splitting.
Point $4)$ follows easily from $3)$.
\end{proof}
We now show an example of a finite, étale and linearly reductive group
$G$ over $\Q$ with $\widetilde{\Delta}(\GCov)\neq\Lex_{R,\text{reg}}^{G}$
(see Theorem \ref{A}).
\begin{example}
\label{rem: counterexample for gcovers and monoidal} Consider $R=\Q$,
$G=\Z/3\Z$, $\alA=\overline{\Q}[x,y]/(x,y)^{2}$ with the action
of $G\times\overline{\Q}\simeq\mu_{3}$ given by $\deg x=\deg y=1$
and $\Gamma=\Omega^{\alA}=(\alA\otimes_{\Q}-)^{G}\colon\Loc^{G}\Q\arr\Loc\overline{\Q}$.
We have that $\alA\notin\GCov(\overline{\Q})=\mu_{3}\textup{-Cov}(\overline{\Q})$
because $\alA$ is not isomorphic to the regular representation (it
does not contain the $\mu_{3}$-representation corresponding to the
character $2\in\Z/3\Z$). On the other hand we have $\Gamma\in\Lex_{\Q,\text{reg}}^{G}(\overline{\Q})$:
the rank condition can be easily checked on the two irreducible representations
of $G$ over $\Q$. By \ref{thm:sheaves and linear functors} we can
conclude that $\Gamma$ is not in the essential image of the functor
$\widetilde{\Delta}\colon\GCov\arr\Lex_{R}^{G}$.
\end{example}
The problem in the above example is that the group $\Z/3\Z$ has a
two-dimensional irreducible representation over $\Q$ which splits
over $\overline{\Q}$. We want therefore to find a class of linearly
reductive groups whose ``irreducible'' representations are also
geometrically irreducible.
\begin{lem}
\label{lem: glrg} Let $I$ be a finite collection of sheaves in $\Loc^{G}R$
which have positive rank in all points of $\Spec R$. The following
are equivalent:
\begin{enumerate}
\item the natural maps
\[
\eta_{M}\colon\bigoplus_{V\in I}V\otimes_{R}\Hom_{R}^{G}(V,M)\arr M\text{ for }M\in\Mod^{G}R
\]
are isomorphisms.
\item for all geometric points $\Spec k\arrdi{\xi}\Spec R$ the set $\{V\otimes_{R}k\}_{V\in I}$
is a set of representatives of the irreducible representations of
$G\times k$ and $V\otimes_{R}k\simeq W\otimes_{R}k$ if and only
if $V=W$.
\item (assuming $\Spec R$ connected) there exists a geometric point $\Spec k\arrdi{\xi}\Spec R$
for which the set $\{V\otimes_{R}k\}_{V\in I}$ is a set of representatives
of the irreducible representations of $G\times k$ and $V\otimes_{R}k\simeq W\otimes_{R}k$
if and only if $V=W$.
\end{enumerate}
In the above cases we have that $\Hom^{G}(V,W)=0$ if $V\neq W\in I$
and $\Hom^{G}(V,V)=R\id_{V}$ if $V\in I$.\end{lem}
\begin{proof}
We are going to use that taking invariants commutes with arbitrary
base changes (see \ref{prop:properties linearly reductive groups}).
If $\Spec k\arr\Spec R$ is a geometric point we set $G_{k}=G\times k$.

$1)\then2)$. If $\Spec k\arr\Spec R$ is a geometric point and $M\in\Mod^{G_{k}}k$
then $\Hom_{R}^{G}(V,M)\simeq\Hom_{k}^{G_{k}}(V\otimes k,M)$ and
$\eta_{M}\simeq(\eta_{M})\otimes k$. Thus we can assume that $R$
is an algebraically closed field. In this case the result follows
by decomposing representations into irreducible ones. 

$2),3)\then1)$. If $V,W\in\Loc^{G}R$ then $\Hom^{G}(V,W)$ is locally
free by \ref{prop:properties linearly reductive groups}, $2)$. Thus,
checking the rank on the geometric points (on the given geometric
point if $\Spec R$ is connected), if $V,W\in I$ then $\Hom^{G}(V,W)=0$
for $V\neq W$ and $\Hom^{G}(V,V)=R\id_{V}$. In particular if $\Spec k\arrdi{\xi}\Spec R$
is any geometric point then $\xi^{*}\colon I_{G}\arr\Loc^{G}k$ is
injective onto a subset of representatives of the irreducible representations
of $G\times k$. Given $M\in\Mod^{G}R$ we therefore have that $\xi^{*}\eta_{M}$
is injective and, if $\xi^{*}(I_{G})$ is a full set of representatives
of irreducible representations of $G\times k$, an isomorphism. If
$\Spec R$ is connected, so that $R[G]$ has constant rank, applying
this consideration to $M=R[G]$ and using \ref{rem: crucial isomorphism}
we can conclude that $3)\then2)$ by dimension. In particular $\eta_{M}$
is an isomorphism on all geometric points of $\Spec R$. If $M$ is
an arbitrary direct sum of locally free $G$-comodules of finite rank
it follows that $\eta_{M}$ is an isomorphism. In general, using \ref{rem: crucial isomorphism},
we can find an exact sequence of $G$-comodules $V_{1}\arr V_{0}\arr M\arr0$
where the $V_{i}$ are sum of copies of $\duale{R[G]}$. Since $\eta_{V_{0}},\eta_{V_{1}}$
are isomorphisms, by functoriality it follows that $\eta_{M}$ is
an isomorphism as well.\end{proof}
\begin{rem}
\label{rem:normalized collection glrg} If $I$ is a collection of
sheaves satisfying the conditions in \ref{lem: glrg}, then there
exists another collection $I'$ satisfying the same conditions and
such that $R\in I$. Indeed notice first that, if $R=R_{1}\times R_{2}$
and we are able to replace the collections $I_{|\Spec R_{1}}$ and
$I_{|\Spec R_{2}}$ then we can easily replace the collection $I$.
In particular, since the map $\eta_{R}$ in \ref{lem: glrg} is an
isomorphism, we can assume there exists $V\in I$ such that $V\otimes\Hom^{G}(V,R)\arr R$
is an isomorphism, which means that $V$ is an invertible sheaf with
the trivial action of $G$. If we replace $V$ by $R$ in $I$ we
find the desired collection. \end{rem}
\begin{defn}
We will say that $G$ has a \emph{good representation theory} over
$R$ if it admits a collection $I$ as in \ref{lem: glrg}. A good
linearly reductive group is a pair $(G,I_{G})$ where $G$ is a finite,
flat, finitely presented and linearly reductive group scheme over
$R$ and $I_{G}$ is a collection as in \ref{lem: glrg}. We will
simply write $G$ if this will not lead to confusion. For simplicity
we will also assume that $R\in I_{G}$ (see \ref{rem:normalized collection glrg}).
\end{defn}
If $R\arr R'$ is a morphism and $G$ is a good linearly reductive
group, then $G\times R'$ is naturally a good linearly reductive group
with the collection of the pullbacks of the modules in $I_{G}$.
\begin{rem}
All diagonalizable group schemes are good over the integers, while
if $R$ is a field, then $G$ is good if and only if its irreducible
representations are geometrically irreducible. 
\end{rem}
We are going to prove that any linearly reductive group is fppf locally
good.
\begin{lem}
\label{lem:deformation of locally free sheaves}Let $\stX$ be a proper
and flat algebraic stack over a Noetherian local ring $R$. Denote
by $k$ the residue field of $R$ and consider a locally free sheaf
$V_{0}$ of rank $n$ over $\stX\times k$. If $\Hl^{2}(\stX\times k,\Endsh(V_{0}))=0$,
then there exists a locally free sheaf of rank $n$ over $\stX\times\widehat{R}$
lifting $V_{0}$, where $\widehat{R}$ is the completion of $R$.\end{lem}
\begin{proof}
Taking into account Grothendieck's existence theorem for proper stacks,
we can assume that $R$ is an Artinian ring (so that $\widehat{R}\simeq R$)
and that we have a lifting $\overline{V}$ of $V_{0}$ over $\stX\times(R/I)$,
where $I$ an ideal of $R$ such that $I^{2}=0$. Define the stack
$\stY$ over the small fppf site $\stX_{\text{fppf}}$ of $\stX$
whose objects over $\Spec B\arr\stX$ are locally free sheaves $N$
of rank $n$ over $B$ with an isomorphism $\phi\colon N\otimes(B/IB)\arr\overline{V}\otimes(B/IB)$.
A section of $\stY\arr\stX_{\text{fppf}}$ yields a lifting of $\overline{V}$
on $\stX$. We are going to prove that $\stY$ is a gerbe over $\stX_{\text{fppf}}$
banded by the sheaf of abelian groups $\pi_{*}\Endsh(V_{0})$, where
$\pi\colon\stX\times k\arr\stX$ is the obvious closed immersion.
Since $\Hl^{2}(\stX,\pi_{*}\Endsh(V_{0}))=\Hl^{2}(\stX\times k,\Endsh(V_{0}))=0$
parametrizes those gerbes (see \cite[Chapter IV, §3, Section 3.4]{Giraud1971}),
we can then conclude that $\stY\arr\stX_{\text{fppf}}$ is a trivial
gerbe, which means that it has a section as required.

I claim that $\overline{V}$ is trivial in the fppf topology of $\stX$,
which implies that $\stY\arr\stX_{\text{fppf}}$ has local sections.
Indeed if $B$ is a ring and $P\arr\Spec B/IB$ is a $\Gl_{n}$-torsor
then by standard deformation theory it extends to a smooth map $Q\arr\Spec B$.
In particular, if we base change to $Q$, we can conclude that $P$
over $Q\times(B/IB)$ has a section, which means that it is trivial.

I also claim that two objects of $\stY$ over the same object of $\stX_{\text{fppf}}$
are locally isomorphic. Replacing again locally free sheaves by $\Gl_{n}$-torsors,
given $\Gl_{n}$-torsors $P,Q$ over $\Spec B$, we have to show that
an equivariant isomorphism $P\times(B/IB)\arr Q\times(B/IB)$ locally
extends to an equivariant isomorphism $P\arr Q$. In particular we
can assume that $P$ and $Q$ are both trivial and in this case the
above property follows because $\Gl_{n}(B)\arr\Gl_{n}(B/IB)$ is surjective,
since $\Gl_{n}$ is smooth.

The previous two claims show that $\stY\arr\stX_{\text{fppf}}$ is
a gerbe. We have now to check the banding and therefore to compute
the automorphism group of an object $(N,\phi)\in\stY$ over a ring
$B$. The group $\Aut(\chi)$ consists of the automorphism $N\arrdi{\lambda}N$
inducing the identity on $N/IN$. It is easy to check that the map
\[
\Hom_{B}(N,IN)\arr\Aut\chi\comma\delta\longmapsto\id_{N}+\delta
\]
is an isomorphism of groups. Since $IN=I\otimes_{R}N$ and $N\otimes(B/m_{R}B)\simeq V_{0}\otimes(B/m_{R}B)$
we have
\[
\Hom_{B}(N,IN)=I\otimes\End_{B}(N)\simeq I/I^{2}\otimes\End_{B}(N)\simeq\End_{B/m_{R}B}(V_{0}\otimes(B/m_{R}B))
\]
\end{proof}
\begin{lem}
\label{lem:lifting representation on henselian rings}Assume that
$R$ is a Henselian ring with residue field $k$. The any finite dimensional
representation of $G$ over $k$ lifts to $R$.\end{lem}
\begin{proof}
Since $G$ is finitely presented, we can assume that $R$ is the Henselization
of a scheme of finite type over $\Z$. Since $G$ is linearly reductive,
we have that $\Hl^{2}(\Bi(G\times k),-)=0$ and, viewing $G$-representations
as sheaves over $\Bi G$ and using \ref{lem:deformation of locally free sheaves},
we obtain a lifting of $V$ to a representation over the completion
$\widehat{R}$. We can then conclude using Artin's approximation theorem
over $R$.\end{proof}
\begin{prop}
\label{prop:G has locally a good representation theory}There exists
an fppf covering $\stU=\{U_{i}\arr\Spec R\}_{i\in I}$ such that $G\times_{S}U_{i}$
has a good representation theory over $U_{i}$ for all $i$. If $G$
is étale over $R$ there exists an étale covering with the same property.\end{prop}
\begin{proof}
We start with the case when $R=k$ is a field. The group $G$ is good
after a finite extension of $k$ because an irreducible representation
of $G$ over the algebraic closure of $k$ is always defined over
a finite extension of $k$. Now assume that $G$ is étale. If $k$
is perfect there is nothing to prove. So assume $\car k=p>0$. After
passing to a separable extension of $k$ we can assume that $G$ is
constant of order prime to $p$. So $G$ is defined over $\F_{p}$,
which is perfect and again we have our claim.

Now return to the general case. Since $G$ is finitely presented,
we can assume that $R$ is of finite type over $\Z$. Let $p\in\Spec R$
and $L/k(p)$ an extension such that $G_{L}=G\times L$ is good, with
$L/k(p)$ separable if $G$ is étale. There exists a flat finitely
presented map $h\colon\Spec R'\arr\Spec R$ such that $h^{-1}(p)\simeq\Spec L$.
If $L/k$ is separable we can even assume that $h$ is étale. This
shows that we can assume that $G_{k(p)}=G\times k(p)$ is good. From
\ref{lem:lifting representation on henselian rings} any $G_{k(p)}$
representation lifts to $R_{p}^{h}$, the Henselization of $R_{p}$,
and, since this ring is a direct limit of algebras étale over $R$,
we get the required result.
\end{proof}
Putting together \ref{lem:lifting representation on henselian rings}
and \ref{prop:G has locally a good representation theory} we get:
\begin{thm}
\label{thm:etale linearly reductive over sctrictly are glrg}A constant
linearly reductive group over a strictly Henselian ring has a good
representation theory.
\end{thm}

\begin{rem}
\label{rem: explicit functor sheaves for glrg} If $(G,I_{G})$ is
a good linearly reductive group there is an explicit way to map linear
functors to sheaves, which may be useful in concrete examples. Let
$T$ be an $R$-scheme, set $\L_{R}^{G}(T)$ for the category of $R$-linear
functors $\Loc^{G}R\arr\QCoh T$ and define
\[
\shF_{*}\colon\L_{R}^{G}(T)\arr\QCoh^{G}T\comma\shF_{\Gamma}=\bigoplus_{V\in I_{G}}\duale V\otimes\Gamma_{V}
\]
where the action of $G$ on the $\Gamma_{V}$ is trivial. Using \ref{lem: glrg}
it is easy to see that $\shF_{*}$ is a quasi-inverse of $\Omega^{*}\colon\QCoh^{G}T\arr\L_{R}^{G}(T)$,
$\Omega^{\shG}=(\shG\otimes-)^{G}$, the other natural isomorphism
being
\[
\beta_{U}\colon\Omega_{U}^{\shF_{\Gamma}}\simeq(U\otimes\shF_{\Gamma})^{G}\simeq\bigoplus_{V\in I_{G}}\Hom^{G}(V,U)\otimes\Gamma_{V}\arr\Gamma_{U}\text{ for }\Gamma\in\L_{R}^{G}(T)\comma U\in\Loc^{G}R
\]
The map $\beta_{U}^{-1}\colon\Gamma_{U}\arr(U\otimes\shF_{\Gamma})^{G}$
is uniquely determined by a map $\alpha_{U}\colon\duale U\otimes\Gamma_{U}\arr\shF_{\Gamma}$.
It is easy to see that:
\begin{enumerate}
\item if $U\in I_{G}$ then $\alpha_{U}$ is the inclusion;
\item if $U=U_{1}\oplus U_{2}$ then $\alpha_{U}$ is zero on $\duale{U_{i}}\otimes\Gamma_{U_{j}}$for
$i\neq j\in\{1,2\}$ and coincides with $\alpha_{U_{i}}$ on $\duale{U_{i}}\otimes\Gamma_{U_{i}}$
for all $i=1,2$;
\item if $U=\shH\otimes U'$ for $\shH\in\Loc R$ and $U'\in\Loc^{G}R$
then $\alpha_{U}$ is
\[
\duale U\otimes\Gamma_{U}\simeq\duale{\shH}\otimes\shH\otimes U'\otimes\Gamma_{U'}\arrdi{\text{ev}_{\shH}\otimes\alpha_{U'}}\shF_{\Gamma}
\]
where $\text{ev}_{\shH}\colon\duale{\shH}\otimes\shH\arr R$ is the
evaluation;
\item if $\gamma\colon V\arr U$ is a $G$-equivariant isomorphism then
$\alpha_{V}=\alpha_{U}\circ[(\duale{\gamma})^{-1}\otimes\Gamma_{\gamma}]$.
\end{enumerate}
Using the maps $\alpha_{*}$ (and by going through the definitions)
if $\Gamma$ is a monoidal functor the associated ring structure on
$\shF_{\Gamma}$ is given by
\[
\duale V\otimes\Gamma_{V}\otimes\duale W\otimes\Gamma_{W}\arr\duale{(V\otimes W)}\otimes\Gamma_{V\otimes W}\arrdi{\alpha_{V\otimes W}}\shF_{\Gamma}\text{ for }V,W\in I_{G}
\]
\end{rem}
\begin{proof}
[Proof of Theorem \ref{A}, last sentence.] The functor $\widetilde{\Delta}\colon\LAlg_{R}^{G}\arr\Lex_{R}^{G}$
is well defined thanks to \ref{prop:properties linearly reductive groups}.
It is an equivalence thanks to \ref{thm:sheaves and linear functors}
and the fact that if $\alA\in\QAlg^{G}T$ and $\Omega^{\alA}\in\Lex_{R}^{G}(T)$
then, using \ref{rem: crucial isomorphism}, $\alA\simeq(\alA\otimes R[G])^{G}=\Omega_{R[G]}^{\alA}$
is locally free of finite rank. 

We now show the last equality in the statement. Using notation from
\ref{rem: explicit functor sheaves for glrg}, if $\Gamma\in\Lex_{R,\text{reg}}^{G}(T)$
then $\alA=\shF_{\Gamma}\in\QAlg^{G}T$ is such that $\Gamma\simeq\Omega^{\alA}$.
We can assume that $\Gamma_{V}$ is free of rank $\rk V$ for all
$V\in I_{G}$. In this case $R[G]\otimes\odi T$ and $\alA$ have
the same decomposition in terms of the representations in $I_{G}$
and thus they are isomorphic.

We finally show that $\GCov$ is open and closed in $\LAlg_{R}^{G}$.
This problem is fppf local in the base, thus we can assume that $G$
is a good linearly reductive group thanks to \ref{prop:G has locally a good representation theory}.
In this case $\GCov$ (resp. $\LAlg_{R}^{G}$) corresponds to $\Lex_{R,\text{reg}}^{G}$
(resp. $\Lex_{R}^{G}$) via $\widetilde{\Delta}$ and $\Lex_{R,\text{reg}}^{G}$
is the locus in $\Lex_{R}^{G}$ of functors $\Gamma$ such that $\rk\Gamma_{V}=\rk V$
for all $V\in I_{G}$. Since $I_{G}$ is finite, this is an open and
closed condition, as required.
\end{proof}

\section{Induction from a subgroup for equivariant algebras.}

As in the previous section we fix a base ring $R$ and a flat, finite
and finitely presented group scheme $G$ over $R$.

Let $H$ be an open and closed subgroup scheme of $G$. If $\shF\in\QCoh^{H}T$
we define the induction from $H$ to $G$ of $\shF$, denoted by $\ind_{H}^{G}\shF$,
as $(\shF\otimes R[G])^{H}\in\QCoh^{G}T$. For details and properties
we refer to \cite[Part I, Section 3]{Jantzen2003}. If $\shF$ is
also a quasi-coherent sheaf of algebras, that is $\shF\in\QAlg^{H}T$,
then $\ind_{H}^{G}\shF\in\QAlg^{G}T$, that is it inherits a natural
structure of sheaf of algebras with an action of $G$. The aim of
this section is to prove the following.
\begin{thm}
\label{thm:ind is etale} If $H$ is an open and closed subgroup scheme
of $G$ the functor
\[
\ind_{H}^{G}\colon\LAlg_{R}^{H}\arr\LAlg_{R}^{G}\comma\alA\longmapsto(\alA\otimes R[G])^{H}
\]
is well defined, quasi-affine and étale. The (open) image consists
of those $\alA\in\LAlg_{R}^{G}T$ such that, for all geometric points
$\Spec k\arr T$, there exists a subset of points of $\Spec(\alA\otimes k)$
whose geometric stabilizers are contained in $H\times k$ and whose
$G(k)$-orbits cover the whole $\Spec(\alA\otimes k)$.\end{thm}
\begin{lem}
\label{cor:base change of local ring is local for strictly henselian ring}
Assume that $R$ is a strictly Henselian ring. If $A,B$ are local
$R$-algebras such that $A$ is finite over $R$ and the maximal ideal
of $B$ lies over the maximal ideal of $R$, then $A\otimes_{R}B$
is local.\end{lem}
\begin{proof}
Set $k_{A}\comma k_{B}$ for their residue fields. Since $A\otimes_{R}B$
is finite over $B$ it is enough to note that $k_{A}\otimes_{k_{R}}k_{B}$
is local since $k_{A}/k_{R}$ is purely inseparable.\end{proof}
\begin{lem}
\label{lem:connected components of finite group schemes over Henselian rings}
Assume that $R$ is a strictly Henselian ring and let $X\arr\Spec R$
be a cover with an action of $G.$ Consider the decomposition into
connected components
\[
G=\bigsqcup_{i\in\underline{G}}G_{i}\text{ and }X=\bigsqcup_{j\in\underline{X}}X_{j}
\]
Given $i\in\underline{G}$ and $j\in\underline{X}$ the restriction
of the action $X_{j}\times G_{i}\arr X$ factors through a unique
component $X_{j\star i}$ with $j\star i\in\underline{X}$. The operation
$-\star-\colon\underline{G}\times\underline{G}\arr\underline{G}$
obtained when $X=G$ with the right action of $G$ by multiplication
makes $\underline{G}$ into a group, whose unity $1\in\underline{G}$
is the connected component containing the identity. In general the
association $\underline{X}\times\underline{G}\arr\underline{X}$ defines
a right action of $\underline{G}$ on the set $\underline{X}$. Moreover
$G_{1}$ is a subgroup scheme of $G$ and the map $G_{i}\times G_{1}\arr G_{i}$
makes $G_{i}$ into a $G_{1}$-torsor for all $i\in\underline{G}$.\end{lem}
\begin{proof}
Finite algebras over Henselian rings are products of their localizations.
In particular the $G_{i}$ and $X_{j}$ are the spectrum of the localizations
of $\Hl^{0}(\odi G)$ and $\Hl^{0}(\odi X)$ respectively. All the
conclusions follow easily from \ref{cor:base change of local ring is local for strictly henselian ring}.\end{proof}
\begin{lem}
\label{lem:criterion ind isomorphism} Let $H$ be an open and closed
subgroup scheme of $G$ and let $B$ be a local ring with residue
field $k$, $\alA\in\LAlg^{G}B$, $Z=\Spec\widetilde{\alA}\subseteq\Spec\alA$
be an $H$-equivariant open and closed subscheme. Then the map $\alA\arr\ind_{H}^{G}\widetilde{\alA}$
induced by the projection $\alA\arr\widetilde{\alA}$ is an isomorphism
if and only if 
\[
(Z\times\overline{k})g\cap Z\times\overline{k}\neq\emptyset\then g\in H(\overline{k})\qquad\forall g\in G(\overline{k})
\]
and the $G(\overline{k})$-orbits of $Z\times\overline{k}$ cover
the whole $\Spec(\alA\otimes\overline{k})$. In this case $\widetilde{\alA}\in\LAlg^{H}B$
and the geometric stabilizers of $Z$ for the action of $H$ or $G$
coincide. If in addition $G$ is étale over $B$, then we can replace
$\overline{k}$ with the separable closure of $k$ in the formula
above.\end{lem}
\begin{proof}
It is easy to see that there exists a (étale if $G/R$ is étale) cover
$\Spec R'\arr\Spec R$ such that $G\times R'$ splits as disjoint
union of copies of $H\times R'$, that is the right cosets of $H\times R'$.
Localizing in a maximal ideal of $R'$ we see that we can assume this
decomposition holds also for $R$ and that $R=B$. In particular $R[G]\simeq R[H]^{\shR}$,
where $\shR\subseteq G(R)$ is a set of representatives of the right
cosets of $H$, and therefore, using \ref{rem: crucial isomorphism},
we have
\[
\ind_{H}^{G}\widetilde{\alA}=(\widetilde{\alA}\otimes R[G])^{H}\simeq(\widetilde{\alA}\otimes R[H]^{\shR})^{H}\simeq((\widetilde{\alA}\otimes R[H])^{H})^{\shR}\simeq\widetilde{\alA}^{\shR}
\]
In particular $\ind_{H}^{G}\widetilde{\alA}$ is flat over $B$ and,
if $\alA\simeq\ind_{H}^{G}\widetilde{\alA}$, then $\widetilde{\alA}$
is locally free and therefore $\widetilde{\alA}\in\LAlg^{H}B$. Since
the map $\alA\arr\ind_{H}^{G}\widetilde{\alA}$ is an isomorphism
if and only if it is so after tensoring with $\overline{k}$ or the
separable closure $k^{s}$, we can assume that $R=B=L$ is $k^{s}$
if $G/B$ is étale or $\overline{k}$ otherwise. The action of $G$
on $\ind_{H}^{G}\widetilde{\alA}\simeq\widetilde{\alA}^{\shR}$ is
induced by the right action of $G(L)$ on $\shR$ and the the action
of $H$ on $\widetilde{\alA}$. Thus the map 
\[
\Spec(\ind_{H}^{G}\widetilde{\alA})=\bigsqcup_{g\in\shR}Z\arr\Spec\alA
\]
is the disjoint union of the $g_{|Z}\colon Z\arr\Spec\alA$ where
$g_{|Z}$ is the restriction of the action of $g\in G(L)$. Taking
into account \ref{lem:connected components of finite group schemes over Henselian rings},
the above map is an isomorphism if and only if $\Spec\alA$ is the
disjoint union of the $Zg$ for $g\in\shR$, which is equivalent to
the two conditions given in the statement.\end{proof}
\begin{defn}
If $R$ is a strictly Henselian ring, $X\arr\Spec R$ a cover with
an action of $G$ and $X_{i}$ a connected component of $X$ we call
the \emph{stabilizer} of $X_{i}$ the open and closed subgroup $H$
of $G$ which is the disjoint union of the components $G_{j}$ of
$G$ such that $X_{i}G_{j}\subseteq X_{i}$.\end{defn}
\begin{lem}
\label{lem:induction from an open sugroup and stabilizers} Assume
that $R$ is a strictly Henselian ring with residue field $k$ and
let $\alA\in\LAlg^{G}R$, $p\in\Spec\alA$ be a maximal ideal and
denote by $H_{p}$ the geometric stabilizer of $p$ and by $U_{p}$
the stabilizer of the connected component $\Spec\alA_{p}$. Then $H_{p}$
is a closed subgroup scheme of $U_{p}\times\overline{k}$, they are
topologically equal and, if $G(\overline{k})$ acts transitively on
$\Spec(\alA\otimes\overline{k})$, there exists an isomorphism
\[
\ind_{U_{p}}^{G}\alA_{p}\simeq\alA
\]
\end{lem}
\begin{proof}
We are going to use \ref{cor:base change of local ring is local for strictly henselian ring}
several times. Set $X=\Spec\alA$ and $X_{p}=\Spec\alA_{p}$. Notice
that the closed points of $\Spec\alA$ correspond to $\Spec(\alA\otimes k)$
or $\Spec(\alA\otimes\overline{k})$, so that we can also think $p\in\Spec(\alA\otimes\overline{k})$.
Moreover $U_{p}\times\overline{k}$ is the stabilizer of the connected
component $\Spec\alA_{p}\otimes\overline{k}$ of $\Spec\alA\otimes\overline{k}$.
In particular $H_{p}(\overline{k})=U_{p}(\overline{k})$ so that $H_{p}$
is a closed subgroup scheme of $G\times\overline{k}$ contained in
$U_{p}\times\overline{k}$. Moreover we can apply \ref{lem:criterion ind isomorphism}
with $Z=\Spec\alA_{p}$ and $H=U_{p}$ obtaining the desired isomorphism.
\end{proof}

\begin{proof}
[Proof of Theorem \ref{thm:ind is etale}.] Arguing as in the proof
of \ref{lem:criterion ind isomorphism}, we can assume that $G$ is
a disjoint union of copies of $H$, namely its right cosets, obtaining
an isomorphism
\[
\ind_{H}^{G}\alB=(\alB\otimes R[G])^{H}\simeq((\alB\otimes R[H])^{H})^{\shR}\simeq\alB^{\shR}\text{ for }\alB\in\LAlg_{R}^{H}
\]
where $\shR\subseteq G(R)$ is a set of representatives of the right
cosets of $H$ in $G$. This shows that $\ind_{H}^{G}$ is well defined.
Moreover, since it is faithful, it is also representable by algebraic
spaces. We are going to prove that it is étale and separated. By \cite[Appendice A, Theorem A.2]{Laumon1999}
it will follow that it is quasi-affine.

Let $A$ be an $R$-algebra and $\xi\colon\Spec A\arr\LAlg_{R}^{G}$
be a map given by $\alA\in\LAlg^{G}A$. The fiber product $X\colon(\Sch/A)^{\op}\arr\sets$
of $\xi$ and $\ind_{H}^{G}$ is given by
\[
X(T)=\{(\alB,\psi)\st\alB\in\LAlg^{H}T\text{ and }\psi\colon\alA\otimes\odi T\simeq\ind_{H}^{G}\alB\}
\]
Notice that the datum $\psi$ can also be given as an $H$-equivariant
map $\alA\otimes\odi T\arr\alB$ which induces an isomorphism $\alA\otimes\odi T\arr\ind_{H}^{G}\alB$
via adjunction. In particular we obtain a map $X\arr\Hilb_{\Spec\alA/A}$
which is a monomorphism because if $(\alB,\psi)\in X$ then the action
of $H$ on $\alB$ is completely determined by the action of $H$
on $\alA$ and by $\psi$. Since $\Hilb_{\Spec\alA/R}$ and monomorphisms
are separated, it follows that $X$ is separated too.

Since $\LAlg_{R}^{H}$ and $\LAlg_{R}^{G}$ are locally of finite
presentation by \ref{prop:LAlg of finite presentation} so is $X\arr\Spec A$.
Thus in order to show that $X$ is étale over $A$ we can assume that
$A$ is an Artinian local ring and prove that, if $J$ is a square
zero ideal of $A$, then an object $(\alB',\psi')\in X(A/J)$ extends
uniquely to $X(A)$. The map $\Spec\alB'\arr\Spec\alA/J\alA$ induced
by $\psi'$ is an $H$-invariant open and closed subscheme of $\Spec\alA/J\alA$.
This gives an open and closed subscheme $\Spec\alB\subseteq\Spec\alA$.
This is also $H$-invariant: if $\gamma\colon\Spec\alB\times H\arr\Spec\alA$
is the restriction of the action, then $\gamma^{-1}(\Spec\alA-\Spec\alB)=\emptyset$
because it is empty after tensoring by $A/J$. Thus we have extended
the $H$-equivariant map
\[
\alA\otimes A/J\arrdi{\psi}\ind_{H}^{G}\alB'\arr\alB'
\]
to an $H$-equivariant map $\alA\arr\alB$ and it is also clear that
this extension is unique up to a unique isomorphism. Finally the map
$\alA\arr\ind_{H}^{G}\alB$ is an isomorphism because it is so after
tensoring by $A/J$.

It remains to characterize the image of $\ind_{H}^{G}$. Let $k$
be an algebraically closed field and $\alA\in\LAlg^{G}k$. Given $p\in\Spec\alA$
we denote by $H_{p}$ its geometric stabilizer and by $U_{p}$ the
stabilizer of $\Spec\alA_{p}$.

Assume that $\alA$ is in the image, that is $\alA\simeq\ind_{H}^{G}\alB$.
The conclusion follows applying \ref{lem:criterion ind isomorphism}
with $\widetilde{\alA}=\alB$. Conversely assume there is a set of
points $Z\subseteq\Spec\alA$ as in the statement. Set $X=\Spec\alA$
and $X_{p}=\Spec\alA_{p}$ for $p\in\Spec\alA$. We can assume that
the points of $Z$ are all in different orbits, that is
\[
X=\bigsqcup_{p\in Z}X_{p}G(k)
\]
By \ref{lem:criterion ind isomorphism} we have $U_{p}(k)=H_{p}(k)$
and therefore $U_{p}\subseteq H$. Moroever we also have
\[
\alA\simeq\prod_{p\in Z}\ind_{U_{p}}^{G}\alA_{p}\simeq\prod_{p\in Z}\ind_{H}^{G}(\ind_{U_{p}}^{H}\alA_{p})\simeq\ind_{H}^{G}(\prod_{p\in Z}\ind_{U_{p}}^{H}\alA_{p})
\]
 as required. 
\end{proof}
We conclude with the following results that will be used in the next
sections.
\begin{cor}
\label{cor:stabilizers and ind} Assume that $G$ is a constant group
and let $\alA\in\LAlg^{G}B$, where $B$ is an $R$-algebra, such
that $\alA^{G}=B$. If $H$ is the geometric stabilizer of a prime
ideal $p$ of $\alA$ lying over $q\in\Spec B$ then there exists
a an étale morphism $B\arr B'$, $q'\in\Spec B'$ over $q$, $\widetilde{\alA}\in\LAlg^{H}B'$
such that $\widetilde{\alA}{}^{H}=B'$and a $G$-equivariant isomorphism
\[
\alA\otimes_{B}B'\simeq\ind_{H}^{G}\widetilde{\alA}
\]
Moreover we can also assume that $\widetilde{\alA}\otimes\overline{k(q')}$
is local, its maximal ideal lies over $p\in\Spec\alA$ and has geometric
stabilizer equal to $H$ .\end{cor}
\begin{proof}
We are going to prove that $G(\overline{k(q)})$ acts transitively
on $\Spec(\alA\otimes\overline{k(q)})$. Using \ref{cor:base change of local ring is local for strictly henselian ring},
we can find a separable finite extension $L/k$ such that $\Spec(\alA\otimes\overline{k(q)})\arr\Spec(\alA\otimes L)$
is bijective. Moreover there exists a flat and local $B$-algebra
$B'$ with residue field $L$. Since $(\alA\otimes B')^{G}=B'$, by
standard arguments it follows that $G$ (as constant group) acts transitively
on the set of maximal ideals of $\alA\otimes B'$ and thus on $\Spec(\alA\otimes L)$
as required. Now let $\overline{p}\in\Spec(\alA\otimes\overline{k(q)})$
lying over $p\in\Spec\alA$. Since $G$ is constant, the geometric
stabilizer $H$ of $p$ (that is of $\overline{p}$) coincides with
the stabilizer of the connected component $\Spec((\alA\otimes\overline{k(q)})_{\overline{p}})$
and, if we set $\overline{\alB}=(\alA\otimes\overline{k(q)})_{\overline{p}}$,
by \ref{lem:induction from an open sugroup and stabilizers} we get
an isomorphism
\[
\alA\otimes\overline{k(q)}\simeq\ind_{H}^{G}\overline{\alB}
\]
Since $\ind_{H}^{G}\colon\LAlg_{R}^{H}\arr\LAlg_{R}^{G}$ is étale,
there exists an étale morphism $\Spec B'\arr\Spec B$, $q'\in\Spec B'$
over $q$, $\alB\in\LAlg^{H}B'$ such that $\alA\otimes B'\simeq\ind_{H}^{G}\alB$
and $\alB\otimes\overline{k(q')}\simeq\overline{\alB}$. Moreover
we have isomorphisms
\[
B'\simeq(\alA\otimes B')^{G}\simeq(\ind_{H}^{G}\alB)^{G}\simeq\alB^{H}
\]
Thus $\widetilde{\alA}=\alB$ satisfies the desired conditions. \end{proof}
\begin{lem}
\label{lem:induction is restriction} Let $H$ be an open and closed
subgroup of $G$, $T$ an $R$-scheme and $\shF\in\QAlg^{H}T$. Then
\[
\Omega^{\ind_{H}^{G}\shF}\simeq\Omega^{\shF}\circ\R_{H}\colon\Loc^{G}R\arr\QCoh T
\]
where $\R_{H}\colon\Loc^{G}R\arr\Loc^{H}R$ is the restriction.\end{lem}
\begin{proof}
Given $V\in\Loc^{G}R$ we have
\[
\Omega_{V}^{\ind_{H}^{G}\shF}=\Hom^{G}(\duale V,\ind_{H}^{G}\shF)\simeq\Hom^{H}(\R_{H}(V)^{\vee},\shF)=\Omega_{\R_{H}(V)}^{\shF}
\]

\end{proof}

\section{Reducibility of $\protect\GCov$ for nonabelian linearly reductive
groups.}

The aim of this section is to prove the reducibility of $\GCov$ when
$G$ is a nonabelian linearly reductive group, that is Theorem \ref{B}.
We fix a base ring $R$ and a finite, flat, finitely presented and
linearly reductive group scheme $G$ over $R$.
\begin{defn}
\label{def:universally reducible}Let $S$ be a scheme and $\stX$
be an algebraic stack over $S$. The stack $\stX$ is called \emph{universally
reducible} over $S$ if, for all base changes $S'\arr S$, the stack
$\stX\times_{S}S'$ is reducible. \end{defn}
\begin{rem}
\label{rem:universally reducible over fibers}It is easy to check
that $\stX$ is universally reducible over $S$ if and only if for
all fields $k$ and maps $\Spec k\arr S$ the fiber is reducible.
\end{rem}
We start by stating the generalization of Theorem \ref{B} we are
going to prove at the end of this section.
\begin{thm}
\label{thm:GCov reducible when G not abelian}If $G$ is a finite,
flat and finitely presented nonabelian and linearly reductive group
scheme over $R$ then $\GCov$ is reducible. If, moreover, $G$ is
defined over a connected scheme, then $\GCov$ is also universally
reducible.
\end{thm}
Note that, if we do not assume that the base $\Spec R$ is connected,
we can not conclude that $\GCov$ is universally reducible, since
one can always take $G$ as disjoint union of $\mu_{2}$ and $S_{3}$
over $\Spec\Q\sqcup\Spec\Q$. On the other hand what happens when
the base is not connected is clear from the following Proposition.
\begin{prop}
\label{prop:the locus where G abelian open and closed}The locus of
$\Spec R$ where $G$ is abelian is open and closed in $\Spec R$.\end{prop}
\begin{proof}
Denote by $Z$ this locus and set $S=\Spec R$. Topologically, $|Z|$
is closed in $S$, because it is the locus where the maps $G\times G\arr G$
given by $(g,h)\longmapsto gh$ and $(g,h)\longmapsto hg$ coincide
and $G$ is flat and proper. We have to prove that, given an algebraically
closed field $k$ and a map $\Spec k\arrdi pS$ such that $G_{k}=G\times k$
is abelian, there exists a fppf neighborhood of $S$ around $p$ where
$G$ is abelian. By \cite[Theorem 2.19]{Abramovich2007}, we can assume
that $G=\Delta\ltimes H$, where $\Delta$ is diagonalizable and $H$
is constant. If $G_{k}$ is abelian, then $H$ is abelian, the map
$H\arr\Aut\Delta\simeq\Aut(\Hom(\Delta,\Gm))^{op}$ is trivial and
therefore $G\simeq\Delta\times H$ is abelian.\end{proof}
\begin{defn}
\label{def:schematic closure} We say that an open substack $\stU$
of an algebraic stack $\stX$ is \emph{schematically dense} if $\stX$
is the only closed substack of $\stX$ containing $\stU$. If $\stU$
is a quasi-compact open substack of $\stX$ its schematic closure
is the minimum of the closed substacks of $\stX$ containing $\stU$
or, alternatively, the (unique) closed substack $\stZ$ of $\stX$
such that $\stU\subseteq\stZ$ and $\stU$ is schematically dense
in $\stZ$.

We denote by $\stZ_{G}$ the schematic closure of $\Bi G$ inside
$\GCov$ and we call it the main irreducible component of $\GCov$.
\end{defn}
The existence of the schematic closure as stated above and the fact
that it is stable by flat base changes follows from \cite[Theorem 11.10.5]{EGAIV-3}.
Although we have called $\stZ_{G}$ the main irreducible component
of $\GCov$, the stack $\stZ_{G}$ is irreducible if and only if $\Spec R$
is irreducible, because this is the only case in which $\Bi G$ is
irreducible.
\begin{lem}
\label{prop:ind B in Z if B in Z} Let $H$ be an open and closed
subgroup scheme of $G$ and $\alB\in\LAlg_{R}^{H}$. Then 
\[
\ind_{H}^{G}\alB\in\Bi G\iff\alB\in\Bi H\comma\ind_{H}^{G}\alB\in\stZ_{G}\iff\alB\in\stZ_{H}
\]
\end{lem}
\begin{proof}
The fact that $\alB\in\Bi H\then\ind_{H}^{G}\alB\in BG$ is well known.
For the converse set $P=\Spec\alB$ and consider it as a sheaf of
sets over $\Sch/T$ with a right action of $H$, where $T$ is the
$R$-scheme over which $\alB$ is defined. Then $Q=\Spec(\ind_{H}^{G}\alB)$
is by definition $(P\times G)/H$, where the $H$ action on $P\times G$
is given by $(p,g)h=(ph,h^{-1}g)$ and the $G$-action is on the right.
It is easy to check that the natural map $P\arr Q$, $p\longmapsto(p,1)$
is an $H$-equivariant monomorphism. Assume that $Q$ is a $G$-torsor.
It follows that $H$ acts freely on $P$, so that sheaf quotient $P/H$
and stack quotient $[P/H${]} coincide. Moreover $P/H\arr Q/G$ is
an isomorphism, so that $P/H\simeq Q/G\simeq T$ because $Q$ is a
$G$-torsor. In conclusion $P\arr[P/H]\simeq T$ is an $H$-torsor.

Since $\HCov$ (resp. $\GCov$) is closed in $\LAlg_{R}^{H}$ (resp.
$\LAlg_{R}^{G}$) by Theorem \ref{A}, it follows that $\stZ_{H}$
(resp. $\stZ_{G}$) is the schematic closure of $\Bi H$ (resp. $\Bi G$)
inside $\LAlg_{R}^{H}$ (resp. $\LAlg_{R}^{G}$). The second equivalence
therefore follows because flat maps preserve schematic closures and
$\ind_{H}^{G}\colon\LAlg_{R}^{H}\arr\LAlg_{R}^{G}$ is étale by \ref{thm:ind is etale}.\end{proof}
\begin{defn}
Assume that $G$ is a good linearly reductive group and that $\Spec R$
is connected. Given a scheme $T$, we will say that a functor $\Omega\colon\Loc^{G}R\arr\Loc T$
(a sheaf of algebras $\alA\in\LAlg^{G}T)$ has \emph{equivariant constant
rank }(or is of equivariant constant rank) if for all $V\in\Loc^{G}R$
the locally free sheaf $\Omega_{V}$ ($\Omega_{V}^{\alA}=(V\otimes\alA)^{G}$)
has constant rank. In this case we define the rank function $\rk^{\Omega}\colon I_{G}\arr\N$
($\rk^{\alA}\colon I_{G}\arr\N$) as
\[
\rk_{V}^{\Omega}=\rk\Omega_{V}\comma(\rk_{V}^{\alA}=\rk_{V}^{\Omega^{\alA}}=\rk(V\otimes\alA)^{G})
\]
Given $f\colon I_{G}\arr\N$ we will still call $f$ the extension
$f\colon\Loc^{G}R\arr\N$ given by
\[
f_{U}=\sum_{V\in I_{G}}\rk(\Hom^{G}(V,U))f_{V}
\]
so that if $\Omega\colon\Loc^{G}R\arr\Loc T$ is an $R$-linear functor
then $\rk_{V}^{\Omega}=\rk\Omega_{V}$ for all $V\in\Loc^{G}R$.\end{defn}
\begin{lem}
\label{lem:solvable if all subgroups are abelian}\cite{Miller1903}
A constant group whose proper subgroups are abelian is solvable.%

\end{lem}
We are ready for the proof of Theorem \ref{thm:GCov reducible when G not abelian}.
\begin{proof}
[Proof of Theorem\emph{ }\ref{thm:GCov reducible when G not abelian}.]
If the base scheme is not connected, then clearly $\GCov$ is reducible.
By \ref{rem:universally reducible over fibers} and \ref{prop:the locus where G abelian open and closed},
we can assume that $S=\Spec k$, where $k$ is a field. Notice that
$\GCov$ is reducible if and only if $\stZ_{G}(\overline{k})\subsetneq\GCov(\overline{k})$,
where $\overline{k}$ is the algebraic closure of $k$. Moreover $\stZ_{G\times\overline{k}}\simeq\stZ_{G}\times\overline{k}$.
Thus, taking into account \ref{prop:the locus where G abelian open and closed},
we can assume that $k$ is algebraically closed, so that $G$ is a
good linearly reductive nonabelian group scheme. 

Let $H$ be an open and closed subgroup of $G$. We claim that if
one of the following statement holds then $\GCov$ is reducible:
\begin{enumerate}
\item $\HCov$ is reducible
\item there exists $f\colon I_{H}\arr\N$ whose extension $f\colon\Loc^{H}k\arr\N$
is such that $f_{\R_{H}V}=\rk V$ for any $V\in I_{G}$ and there
exists $\Delta\in I_{H}$ such that $f_{\Delta}\neq\rk\Delta$
\end{enumerate}
Assume that $\HCov$ is reducible and, by contradiction, that $\GCov$
is irreducible. If $B\in\HCov(k)$ then $\ind_{H}^{G}B\in\GCov(k)=\stZ_{G}(k)$
and so $B\in\stZ_{H}(k)$ by \ref{prop:ind B in Z if B in Z}. Therefore
$\HCov$ is irreducible.

Now let $f\colon I_{H}\arr\N$ as in $2)$ and define 
\[
F=\bigoplus_{R\neq\Delta\in I_{H}}\duale{\Delta}\otimes k^{f_{\Delta}}\comma B=k\oplus F
\]
so that $f=\rk^{B}$ (note that by hypothesis we have $f_{R}=1$).
Setting $F^{2}=0$ we obtain a structure of algebra on $B$ such that
$B\in\LAlg^{H}k$. We claim that $A=\ind_{H}^{G}B\in(\GCov(k)-\stZ_{G}(k))$.
Indeed we have $\Omega^{A}=\Omega^{B}\circ\R_{H}$ by \ref{lem:induction is restriction},
so that
\[
\rk\Omega_{V}^{A}=\rk\Omega_{\R_{H}V}^{B}=f_{\R_{H}V}=\rk V\text{ for all }V\in\Rep^{G}R
\]
Thus $\Omega^{A}\in\Lex_{R,\text{reg}}^{G}$ and, since $G$ is good,
by Theorem \ref{A} we can conclude that $A\in\GCov$. If by contradiction
$A\in\stZ_{G}(k)$, by \ref{prop:ind B in Z if B in Z} we have $B\in\stZ_{H}(k)\subseteq\HCov(k)$
so that, by Theorem \ref{A}, $\rk\Omega_{\Delta}^{B}=f_{\Delta}=\rk\Delta$
for all $\Delta\in I_{H}$, which is not the case.

We return now to the original statement. We are going to use notation
from \ref{lem:connected components of finite group schemes over Henselian rings}.
By \cite[Theorem 2.19]{Abramovich2007} we have $G=G_{1}\ltimes\underline{G}$
with $G_{1}$ diagonalizable. In particular $\underline{G}$ cannot
be trivial. If $\underline{G}$ is not solvable take a minimal nonabelian
subgroup $K$ of $\underline{G}$. All the proper subgroups of $K$
are abelian and therefore $K$ is solvable thanks to \ref{lem:solvable if all subgroups are abelian}.
If we call $\phi\colon G\arr\underline{G}$ the natural projection,
then $G'=\phi^{-1}(K)$ is a nonabelian open and closed subgroup of
$G$ such that $\underline{G'}\simeq K$ is solvable. Using situation
$1)$ above we can replace $G$ by $G'$, that is assume that $\underline{G}$
is solvable. In particular there exists a surjective homomorphism
$\alpha\colon G\arr\Z/p\Z$ for some prime $p$. Set $H=\Ker\alpha$,
which is an open and closed subgroup of $G$. If $H$ is nonabelian,
using again situation $1)$ we can replace $G$ by $H$. Proceeding
by induction we can finally assume to have a surjection $G\arr\Z/p\Z$
whose kernel $H$ is abelian. Since $H$ is linearly reductive and
$k$ is algebraically closed the group $H$ is diagonalizable. Set
$N=\Hom(H,\Gm)$. We will construct an $f\colon I_{H}\arr\N$ as in
situation $2)$ above. This will conclude the proof.

Since $H$ is commutative, the group $G/H\simeq\Z/p\Z$ acts on $H$
and on $N=\Hom(H,\Gm)$ by conjugation. Given $m\in N$ we are going
to denote by $V_{m}$ the corresponding one-dimensional representation
of $H$. Let $\shR\subseteq N$ be a set of representatives of $N/(\Z/p\Z)$.
Note that, since $p$ is prime, an element $n\in N$ is fixed or its
orbit $o(n)$ has order $p$. We claim that if $V\in I_{G}$ there
exists a unique $m\in\shR$ such that 
\[
\R_{H}V=V_{m}^{\rk V}\text{ with }|o(m)|=1\text{ or }V=\ind_{H}^{G}V_{m}\text{ with }|o(m)|=p
\]
Indeed there exists $m\in N$ such that $V\subseteq\ind_{H}^{G}V_{m}$.Given
$n,n'\in N$ we have
\[
\R_{H}\ind_{H}^{G}V_{n}=\bigoplus_{g\in\Z/p\Z}V_{g(n)}\text{ and (}\ind_{H}^{G}V_{n}\simeq\ind_{H}^{G}V_{n'}\iff n'\in o(n))
\]
So we can assume $m\in\shR$. Moreover such an $m$ is unique since
if $V\subseteq\ind_{H}^{G}V_{m'}$, $\R_{H}V$ contains some $V_{n}$
where $n\in N$ is in the orbit of both $m$ and $m'$. In particular,
if $|o(m)|=1$, then $\ind_{H}^{G}V_{m}=V_{m}^{p}$ and therefore
$\R_{H}V=V_{m}^{\rk V}$. So assume $|o(m)|=p$. Given $W\in\Loc^{G}k$
$(\Loc^{H}k)$ and $g\in G(k)$ call $W_{g}$ the representation of
$G$ $(H)$ that has $W$ as underlying vector space, while the action
of $G$ $(H)$ is given by $t\star x=(g^{-1}tg)x$. Note that by definition
$(V_{n})_{g}=V_{g(n)}$. In particular the multiplication by $g^{-1}$
on $V$ yields a $G$-equivariant isomorphism $V\simeq V_{g}$ and
therefore $V_{n}\subseteq\R_{H}V$ implies that $V_{g(n)}\subseteq\R_{H}V$.
Since $|o(m)|=p$ we can conclude that $V=\ind_{H}^{G}V_{m}$. Define
\[
f_{V_{n}}=\left\{ \begin{array}{cc}
|o(n)| & \text{if }n\in\shR\\
0 & \text{otherwise}
\end{array}\right.
\]
We claim that $f$ satisfies the property $2)$. Indeed if $V\in I_{G}$
and there exists $m\in\shR$ such that $V=V_{m}^{\rk V}$ with $|o(m)|=1$
then $f_{\R_{H}V}=\rk Vf_{V_{m}}=\rk V$. Otherwise there exists $m\in\shR$
with $|o(m)|=p$ such that
\[
V=\ind_{H}^{G}V_{m}\then f_{\R_{H}V}=\sum_{g\in\Z/p\Z}f_{V_{g(m)}}=p=\rk V
\]
Finally note that if $n\in\shR$ is such that $|o(n)|=p$ then $f_{V_{n}}=p\neq1=\rk V_{n}$.
So we have to show that such an $n$ exists. If by contradiction this
is false, then the actions of $\Z/p\Z$ on $N$ and $H$, as well
as the action of $G$ on $H$ by conjugation are trivial. So $H$
commutes with all the elements of $G$. Let $g\in G(k)\simeq\underline{G}$
not in $H$, so that it lies over a generator of $G/H\simeq\Z/p\Z$.
If $T$ is a $k$-scheme, any element of $G(T)$ can be written as
$hg^{i}$ with $h\in H(T)$ and $0\leq i<p$. It is straightforward
to check that two such elements commute and that therefore $G$ is
abelian, which is not the case.%

\end{proof}

\section{Regularity in codimension $1$}

The aim of this section is to prove Theorem \ref{C}. In this section
we fix a finite and étale group scheme $G$ over $R$. We require
the étaleness condition on $G$ because we want $G$-torsors to be
regular over a regular base.

We start with some definitions and remarks. In what follows $T$ will
be an arbitrary $R$-scheme if not specified otherwise. %

\begin{rem}
\label{rem: G-torsors and etale maps} If $f\colon X\arr T$ is a
cover with an action of $G$ then $f$ is a $G$-torsor if and only
if $f$ is étale, $X/G=T$ and $\rk f_{*}\odi X=\rk G$. The implication
$\then$ is easy. For the converse, since the locus where $f$ is
a $G$-torsor is open in $T$ and taking invariants commutes with
flat base changes of $T$, we can assume that $T=\Spec B$, where
$B$ is a local ring, that $G$ is constant and that $X$ is a disjoint
union of $\rk G$ copies of $T$. Since $G$ acts transitively on
the closed points of $X$ because $X/G=T$, the orbit map $G\times T\arr X$
is an étale surjective cover. The rank condition implies that this
is an isomorphism.
\end{rem}

\begin{rem}
\label{rem: evaluation and trace} If $G$ is a good linearly reductive
group and $V\in I_{G}$ then $\rk V\in R^{*}$ and the evaluation
map $e_{V}\colon V\otimes\duale V\arr R$ induces an isomorphism $(V\otimes\duale V)^{G}\arr R$.
By a local check we see that $e_{V}$ is surjective and, since $G$
is linearly reductive, we can conclude that $(V\otimes\duale V)^{G}\arr R$
is surjective too. Moreover we have a $G$-equivariant isomorphism
$\Hom_{R}(V,V)\simeq V\otimes\duale V$ and the map $e_{V}$ corresponds
to the trace map $\tr_{V}\colon\Hom_{R}(V,V)\arr R$ under this isomorphism.
Since $\Hom_{R}^{G}(V,V)=R\id_{V}$ by \ref{lem: glrg} we can conclude
that $(V\otimes\duale V)^{G}\arr R$ is an isomorphism and, since
$\tr_{V}(\id_{V})=\rk V$, that $\rk V\in R^{*}$.\end{rem}
\begin{defn}
Let $f\colon X\arr T$ be a cover. The trace map of $f$ will be denoted
by 
\[
\tr_{f}\colon f_{*}\odi X\arr\odi T
\]
We also set
\[
\widetilde{\tr}_{f}\colon f_{*}\odi X\arr\duale{(f_{*}\odi X)}\comma x\longmapsto\tr_{f}(x\cdot-)\text{ and }\shQ_{f}=\Coker(\widetilde{\tr}_{f})\in\QCoh(T)
\]
The discriminant section $s_{f}\in(\det f_{*}\odi X)^{-2}$ is the
section induced by the determinant of the map $\widetilde{\tr}_{f}$. 

Assume now that $G$ acts on $X$ over $T$ and that $X/G=T$ and
consider $V\in\Loc^{G}R$. If $f$ is a $G$-cover or $G$ is linearly
reductive we denote by
\[
\Omega^{f}\colon\Loc^{G}R\arr\Loc T\comma\Omega^{f}=(f_{*}\odi X\otimes-)^{G}
\]
the associated monoidal functor (see Theorem \ref{A}), by
\[
\omega_{f,V}\colon\Omega_{V}^{f}\otimes\Omega_{\duale V}^{f}\arr\Omega_{V\otimes\duale V}^{f}\arr\Omega_{R}^{f}\simeq\odi T
\]
where the first map is given by the monoidality, while the second
is induced by the evaluation $e_{V}\colon V\otimes\duale V\arr R$,
by
\[
\xi_{f,V}\colon\Omega_{\duale V}^{f}\arr(\Omega_{V}^{f})^{\vee}
\]
the induced map and set $\shQ_{f,V}=\Coker(\xi_{f,V})$. If $f$ is
a $G$-cover, then the source and target of the map $\xi_{f,V}$ are
locally free sheaves of the same rank $\rk V$ by Theorem \ref{A},
and we denote by
\[
s_{f,V}\in(\det\Omega_{V}^{f}\otimes\det\Omega_{\duale V}^{f})^{-1}
\]
the section induced by $\det\xi_{f,V}$.

When $\alA\in\LAlg^{G}T$ and $f\colon\Spec\alA\arr T$ we will use
the subscript $-_{\alA}$ instead of $-_{f}$.\end{defn}
\begin{rem}
\label{rem:trace invariant} If $\alA\in\LAlg^{G}T$ then $\tr_{\alA}\colon\alA\arr\odi T$
is $G$-equivariant. Indeed one can assume $T$ is affine, $G$ is
constant and $\alA$ is free and use the invariancy of the trace map
under conjugation.\end{rem}
\begin{lem}
\label{lem:trace for ring with group action}Assume that $R$ is a
local ring, that $G$ is a good linearly reductive group and let $\alA\in\LAlg^{G}T$
be such that $\alA^{G}=\odi T$ and $\rk\alA=\rk G$. Then
\[
\Ker\tr_{\alA}\simeq\bigoplus_{R\neq V\in I_{G}}\duale V\otimes\Omega_{V}^{\alA}\text{ and }\shQ_{\alA}\simeq\bigoplus_{V\in I_{G}}\duale V\otimes\shQ_{\alA,V}
\]
Moreover if $\alA\in\GCov$ then there exists an isomorphism
\[
(\det f_{*}\odi X)^{-2}\simeq\bigotimes_{V\in I_{G}}(\det(\Omega_{V}^{f})^{-1}\otimes\det(\Omega_{\duale V}^{f})^{-1})^{\rk V}\text{ such that }s_{f}\longmapsto\bigotimes_{V\in I_{G}}s_{f,V}^{\otimes\rk V}
\]
\end{lem}
\begin{proof}
Notice that, since $R$ is local, then if $V\in I_{G}$ there exists
a unique $\hat{V}\in I_{G}$ such that $\hat{V}\simeq\duale V$. For
all $V\in I_{G}$ let us fix an equivariant isomorphism $\zeta_{V}\colon V^{\vee}\arr\hat{V}$.
For simplicity set also $\Omega=\Omega^{\alA}\colon\Loc^{G}R\arr\Loc T$.

Since $\tr_{\alA}\colon\alA\arr\odi T$ is $G$-invariant, we have
that $\Ker\tr_{\alA}$ is $G$-invariant too. By \ref{rem: explicit functor sheaves for glrg}
we have 
\[
\Ker\tr_{\alA}=\bigoplus_{V\in I_{G}}\duale V\otimes\Gamma_{V}\text{ with }\Gamma_{V}\subseteq\Omega_{V}
\]
Since $G$ is linearly reductive and $\rk\alA=\rk G$, we have $\tr_{\alA}(1)\in\odi T^{*}$
and, in particular, that $\tr_{\alA}\colon\alA\arr\odi T$ is surjective.
So
\[
\odi T=\bigoplus_{V\in I_{G}}\duale V\otimes(\Omega_{V}/\Gamma_{V})
\]
is a $G$-equivariant decomposition and therefore $\Gamma_{V}=\Omega_{V}$
for $R\neq V\in I_{G}$ and $\Gamma_{R}=0$. In other words $\tr_{\alA}=(\rk G)\pi$,
where $\pi\colon\alA\arr\odi T$ is the projection according to the
$G$-equivariant decomposition of $\alA$. We are going to use the
description given in \ref{rem: explicit functor sheaves for glrg}
of the product of 
\[
\alA=\bigoplus_{V\in I_{G}}\duale V\otimes\Omega_{V}
\]
using the maps $\alpha_{U}\colon\duale U\otimes\Omega_{U}\arr\alA$
for $U\in\Loc^{G}R$. Notice that, given $V,W\in I_{G}$, the product
of elements of $\duale V\otimes\Omega_{V}$ and $\duale W\otimes\Omega_{W}$
lies in $\Ker\tr_{\alA}=\ker\pi$, i.e. has no component in $\alA^{G}\simeq\duale R\otimes\Omega_{R}$,
except for the case $(V\otimes W)^{G}\neq0$. Since
\[
(V\otimes W)^{G}=\Hom^{G}(V,\duale W)
\]
this is the case only when $W=\hat{V}$. So the trace map $\widetilde{\tr}_{\alA}\colon\alA\arr\duale{\alA}$
is the direct sum of the maps 
\[
\overline{\xi}_{V}\colon\duale V\otimes\Omega_{V}\arr\duale{(\duale{(\hat{V})}\otimes\Omega_{\hat{V}})}
\]
induced by $\delta_{V}\colon\duale V\otimes\Omega_{V}\otimes\duale{(\hat{V})}\otimes\Omega_{\hat{V}}\arr\alA\otimes\alA\arr\alA\arrdi{\tr_{\alA}}\odi T$,
which is also the composition
\[
\duale V\otimes\Omega_{V}\otimes\duale{(\hat{V})}\otimes\Omega_{\hat{V}}\simeq\duale{(V\otimes\hat{V})}\otimes\Omega_{V}\otimes\Omega_{\hat{V}}\arr\duale{(V\otimes\hat{V})}\otimes\Omega_{V\otimes\hat{V}}\arrdi{\alpha_{V\otimes\hat{V}}}\alA\arrdi{\rk G\pi}\odi T
\]
Denote by $e_{V}\colon V\otimes\duale V\arr R$ the evaluation map.
By replacing $\hat{V}$ by $\duale V$ using the given isomorphism,
we are going to check that the composition of the last two maps above
is the evaluation $(V\otimes V^{\vee})^{\vee}\simeq V^{\vee}\otimes V\arrdi{e_{V}}R$
tensor $\Omega_{e_{V}}$, up to an invertible element. This will imply
that $\overline{\xi}_{V}$ is isomorphic to the map
\[
\id_{\duale V}\otimes\xi_{\alA,\duale V}\colon\duale V\otimes\Omega_{V}\arr\duale V\otimes\duale{(\Omega_{\duale V})}
\]
and, from this, the claimed result easily follow.

By \ref{rem: evaluation and trace} the map $e_{V}\colon V\otimes\duale V\arr R$
is surjective and it extends to a $G$-equivariant isomorphism $\gamma\colon V\otimes\duale V\arr R\oplus Z$
where $Z\in\Loc^{G}R$ is such that $Z^{G}=0$. By \ref{rem: explicit functor sheaves for glrg}
we have that $\alpha_{V\otimes\duale V}=\alpha_{R\oplus Z}\circ((\duale{\gamma})^{-1}\otimes\Omega_{\gamma})$
and, since $Z^{G}=0$, that $\pi\circ\alpha_{R\oplus Z}\colon(R\oplus Z)^{\vee}\otimes\Omega_{R\oplus Z}\arr\odi T\simeq R^{\vee}\otimes\Omega_{R}$
is the tensor product of the two natural projections. Since $V\otimes V^{\vee}\arrdi{\gamma}R\oplus Z\arr R$
is $e_{V}$, we can conclude that $\pi\circ\alpha_{V\otimes V^{\vee}}$
is the tensor product of $\Omega_{e_{V}}\colon V\otimes\duale V\arr R$
and $(V\otimes V^{\vee})^{\vee}\arrdi{(\gamma^{\vee})^{-1}}(R\oplus Z)^{\vee}\arr\duale R$.
This last map is surjective, $G$-equivariant and therefore it is,
up to an invertible element of $R$, the map $(V\otimes V^{\vee})^{\vee}\simeq V^{\vee}\otimes V\arrdi{e_{V}}R$
by \ref{rem: evaluation and trace}.
\end{proof}

\begin{proof}
[Proof of Theorem \ref{C}] Recall that the loci in $Y$ where $f:X\arr Y$
is a $G$-torsor or a $G$-cover are open thanks to \ref{prop:LAlg of finite presentation}
and that, when $G$ is constant, it acts transitively on the set of
points of $X$ over a given point of $Y$ because $X/G=Y$. In particular
the geometric stabilizers of two points of $X$ over a given point
of $Y$ are conjugates in $G$ and therefore isomorphic. We start
by proving how to deduce the two claims after $3)$. For the first
claim, by $3)$ we have $\rk f=\rk G/rkT$, so that $f$ is generically
a $G$-torsor (that is $T=0$) if and only if $\rk f=\rk G$. Moreover
when $T=0$ the description of the geometric stabilizers of the codimension
$1$ points of $X$ over $q$ is contained in $3)$. For the second
claim it is enough to note that the generic fiber of $X$ is $\Spec L$,
where $L/k(R)$ is a finite field extension with $L^{G}=k(R)$ and
the action of $G$ on $L$ is faithful because $\Aut_{Y}X\arr\Aut_{k(R)}L$
is injective: it follows that $L/k(R)$ is a Galois extension with
group $G$ and therefore $\rk f=\dim_{k(R)}L=\rk G$. 

We start by showing the equivalence between $1)$, $2)$, $3)$ and
the following condition:
\begin{enumerate}
\item [2')] the module $\shQ_{f}\otimes\odi{Y,q}$ is defined over $k(q)$
and the integer $\rk H/\rk T$, where $H$ and $T$ are the geometric
stabilizers of a point of $X$ over $q$ and a generic point of $X$
respectively, is coprime with $\car k(q)$.
\end{enumerate}
We will show that the quotient $\rk H/\rk T$ is an integer. We are
going to use some results and definitions from \cite{Tonini2015}.
In particular all points of $X$ over $q$ are tame with separable
residue fields if and only if the common rank (over $\overline{k(q)}$)
of a connected component of $X\times_{Y}\overline{k(q)}$ is coprime
with $\car k(q)$ (see \cite[Lemma 1.6, Corollary 1.7]{Tonini2015}).
In particular $3)\then1)$: this common rank is $\rk\alB=\rk H/\rk T$
applying \ref{cor:stabilizers and ind} to $\alB\otimes\overline{k(q)}/\overline{k(q)}$.
Moreover we can replace $Y$ by any étale neighborhood around $q$
and, in particular, assume $G$ constant and $Y=\Spec R$.

Write $X=\Spec\alA$ with $\alA\in\LAlg^{G}R$ and let $H$ be the
geometric stabilizer of a point of $\Spec\alA$ over $q$. By \ref{cor:stabilizers and ind}
we can assume $\alA\simeq\ind_{H}^{G}\widetilde{\alA}$ with $\widetilde{\alA}\in\LAlg^{H}R$
such that $\widetilde{\alA}\otimes_{R}\overline{k(q)}$ is local,
$\widetilde{\alA}^{H}=R$ and $H$ is the geometric stabilizer of
the maximal ideal of $\widetilde{\alA}\otimes_{R}R_{q}$. As rings
we have $\alA\simeq\widetilde{\alA}^{(\rk G/\rk H)}$, so that $\shQ_{\alA}\simeq\shQ_{\widetilde{\alA}}^{(\rk G/\rk H)}$,
$s_{\alA}\simeq s_{\widetilde{\alA}}^{(\rk G/\rk H)}$ and $\alA$
is regular in the points over $q$ if and only if the local ring $\widetilde{\alA}\otimes_{R}R_{q}$
is regular. The above discussion shows that we can assume that $\alA\otimes_{R}\overline{k(q)}$
is local and that $G$ is its geometric stabilizer. Let $\overline{G}$
be the image of the map $G\arr\Aut\alA$ and note that all the maps
$\Aut\alA\arr\Aut(\alA\otimes R_{q})\arr\Aut(\alA\otimes k(R))$ are
injective because $\alA$ is a locally free $R$-module. The equivalence
between $1)$, $2)$ and $2')$ can be checked directly on $R_{q}$.
Since being a $\overline{G}$-cover is an open condition, also $1)\then3)$
can be checked on $R_{q}$. Thus we can assume that $R$ is a DVR
(discrete valuation ring), so that $\alA$ is also a local ring. 

Notice that $2)$, $3)$ and $2')$ implies that $\alA/R$ is generically
étale. This also follows from $1)$: if $\alA$ is a domain then $\alA\otimes k(R)$
is a field extension of $k(R)$ with $(\alA\otimes k(R))^{G}=k(R)$
and therefore separable. Thus we can assume that $\alA/R$ is generically
étale so that, by \cite[Corollary 1.7]{Tonini2015}, it follows that
$\alA/R$ is tame with separable residue fields if and only if $\rk\alA$
and $\car k(q)$ are coprime. Since $G$ acts transitively on $Z=\Spec(\alA\otimes\overline{k(R)})$,
it follows that $Z\simeq G/T$ as $G$-space, where $T$ is the geometric
stabilizer of a generic point of $\alA$. In particular $\rk\alA=\rk G/\rk T$,
which is an integer. Thus \cite[Main Theorem]{Tonini2015} exactly
implies the equivalence between the conditions $1)$, $2)$ and $2')$.

It remains to show $1)\then3)$. Since $\alA$ is a domain, $\alA\otimes k(R)$
is a field. Moreover $\overline{G}$ acts faithfully on $\alA\otimes k(R)$
and $(\alA\otimes k(R))^{\overline{G}}=k(R)$. It follows that $\alA\otimes k(R)/k(R)$
is a Galois extension with group $\overline{G}$ and therefore a $\overline{G}$
torsor. It follows that $\Ker(G\arr\overline{G})=T$ is the geometric
stabilizer of the generic point of $\alA$. In particular $\rk\overline{G}$
is coprime with $\car k(q)$, which implies that the map $\overline{G}\arr\Aut\alA\arr\Aut(p/p^{2})\simeq k(p)^{*}$,
where $p$ is the maximal ideal of $\alA$, is injective and therefore
that $\overline{G}$ is cyclic. Thus $\overline{G}$ is linearly reductive
over $R$ and, since $\overline{G}\textup{-Cov}\subseteq\LAlg_{R}^{\overline{G}}$
is closed in this case by Theorem \ref{A} and $\alA/R$ is generically
a $\overline{G}$-torsor, we can conclude that $\alA$ is a $\overline{G}$-cover
over $R$.

We now deal with the last part of the statement. In particular we
assume from now on that $G$ is linearly reductive and $\rk f=\rk G$.
Since $1)$ implies that $f$ is a $G$-cover, more precisely $f\in\stZ_{G}(Y)$,
we will assume $f\in\GCov(Y)$ in what follows.

Denote by $B_{q}$ the strict Henselization of $\odi{Y,q}$, which
is an unramified extension of $\odi{Y,q}$ and a DVR, and by $f_{q}\in\GCov(B_{q})$
the base change of $f$. By \ref{thm:etale linearly reductive over sctrictly are glrg}
the group $G_{q}=G\times B_{q}$ has a good representation theory
over $B_{q}$. Moreover, if $U,W\in\Rep^{G}R$, then $\xi_{f,U\oplus W}=\xi_{f,U}\oplus\xi_{f,W}$,
so that $\shQ_{f,U\oplus W}\simeq\shQ_{f,U}\oplus\shQ_{f,W}$ and
everything commutes with base change. Using \ref{lem:trace for ring with group action}
we obtain
\[
\shQ_{f}\otimes B_{q}\simeq\bigoplus_{V\in I_{G_{q}}}\duale V\otimes\shQ_{f_{q},V}\simeq\shQ_{f,R[G]}\otimes B_{q}
\]
Since for all $U\in\Rep^{G}R$ the representation $U\otimes B_{q}$
splits as a direct sum of representations in $I_{G_{q}}$ we can conclude
that $5)\iff2')$. 

Now notice that, for all $U\in\Rep^{G}R$, the number $v_{q}(s_{f,U})$
coincides with the length of $\shQ_{f,U}\otimes B_{q}$ over $B_{q}$.
In particular, for all $U\in\Rep^{G}R$, if $\shQ_{f,U}\otimes B_{q}$
is defined over $k(q)$ then $v_{q}(s_{f,U})\leq\rk_{q}U$ because
$\shQ_{f,U}\otimes B_{q}$ is a quotient of $(\Omega_{U}^{f})^{\vee}\otimes B_{q}$
which has rank $\rk_{q}U$. Moreover $\xi_{f,R}$ is by construction
an isomorphism so that, if $U\in\Loc^{G}R$, we have $\shQ_{f,U}=\shQ_{f,U/U^{G}}$
and $v_{q}(s_{f,U})=v_{q}(s_{f,U/U^{G}})$ because $U\simeq U^{G}\oplus U/U^{G}$.
Thus $5)\then4)$. Since we have 
\[
v_{q}(s_{f,R[G]})=v_{q}(s_{f})=\sum_{V\in I_{G_{q}}}\rk V\cdot v_{q}(s_{f_{q},V})\text{ and }v_{q}(s_{f,R})=0
\]
we can also conclude that $4)\then2)$.
\end{proof}
\bibliographystyle{amsalpha}
\bibliography{biblio}

\end{document}

%% file: packages_and_functions.for_preamble.tex
\usepackage{extarrow}
\usepackage{enumerate}
\usepackage{fontenc}
\usepackage[T1]{fontenc}
\usepackage{graphicx}
\usepackage[colorlinks=true, linkcolor=blue]{hyperref}
\usepackage{latexsym}
\usepackage{mathrsfs}
\usepackage{mathtext}
\usepackage{tikz}
\usepackage{textcomp}
\usepackage{upgreek}

\usetikzlibrary{arrows}
\usetikzlibrary{matrix}
\usetikzlibrary{shapes}
\usetikzlibrary{snakes}
\usetikzlibrary{matrix}

\DeclareMathOperator{\Add}{\textup{Add}}
\DeclareMathOperator{\Aff}{\textup{Aff}}
\DeclareMathOperator{\Alg}{\textup{Alg}}
\DeclareMathOperator{\Ann}{\textup{Ann}}
\DeclareMathOperator{\Arr}{\textup{Arr}}
\DeclareMathOperator{\Art}{\textup{Art}}
\DeclareMathOperator{\Ass}{\textup{Ass}}
\DeclareMathOperator{\Aut}{\textup{Aut}}
\DeclareMathOperator{\Autsh}{\underline{\textup{Aut}}}
\DeclareMathOperator{\Bi}{\textup{B}}
\DeclareMathOperator{\CAdd}{\textup{CAdd}}
\DeclareMathOperator{\CAlg}{\textup{CAlg}}
\DeclareMathOperator{\CMon}{\textup{CMon}}
\DeclareMathOperator{\CPMon}{\textup{CPMon}}
\DeclareMathOperator{\CRings}{\textup{CRings}}
\DeclareMathOperator{\CSMon}{\textup{CSMon}}
\DeclareMathOperator{\CaCl}{\textup{CaCl}}
\DeclareMathOperator{\Cart}{\textup{Cart}}
\DeclareMathOperator{\Cl}{\textup{Cl}}
\DeclareMathOperator{\Coh}{\textup{Coh}}
\DeclareMathOperator{\Coker}{\textup{Coker}}
\DeclareMathOperator{\Cov}{\textup{Cov}}
\DeclareMathOperator{\Der}{\textup{Der}}
\DeclareMathOperator{\Div}{\textup{Div}}
\DeclareMathOperator{\End}{\textup{End}}
\DeclareMathOperator{\Endsh}{\underline{\textup{End}}}
\DeclareMathOperator{\Ext}{\textup{Ext}}
\DeclareMathOperator{\Extsh}{\underline{\textup{Ext}}}
\DeclareMathOperator{\FAdd}{\textup{FAdd}}
\DeclareMathOperator{\FCoh}{\textup{FCoh}}
\DeclareMathOperator{\FGrad}{\textup{FGrad}}
\DeclareMathOperator{\FLoc}{\textup{FLoc}}
\DeclareMathOperator{\FMod}{\textup{FMod}}
\DeclareMathOperator{\FPMon}{\textup{FPMon}}
\DeclareMathOperator{\FRep}{\textup{FRep}}
\DeclareMathOperator{\FSMon}{\textup{FSMon}}
\DeclareMathOperator{\FVect}{\textup{FVect}}
\DeclareMathOperator{\Fibr}{\textup{Fibr}}
\DeclareMathOperator{\Fix}{\textup{Fix}}
\DeclareMathOperator{\Fl}{\textup{Fl}}
\DeclareMathOperator{\Fr}{\textup{Fr}}
\DeclareMathOperator{\Funct}{\textup{Funct}}
\DeclareMathOperator{\GAlg}{\textup{GAlg}}
\DeclareMathOperator{\GExt}{\textup{GExt}}
\DeclareMathOperator{\GHom}{\textup{GHom}}
\DeclareMathOperator{\GL}{\textup{GL}}
\DeclareMathOperator{\GMod}{\textup{GMod}}
\DeclareMathOperator{\GRis}{\textup{GRis}}
\DeclareMathOperator{\GRiv}{\textup{GRiv}}
\DeclareMathOperator{\Gal}{\textup{Gal}}
\DeclareMathOperator{\Gl}{\textup{Gl}}
\DeclareMathOperator{\Grad}{\textup{Grad}}
\DeclareMathOperator{\Hilb}{\textup{Hilb}}
\DeclareMathOperator{\Hl}{\textup{H}}
\DeclareMathOperator{\Hom}{\textup{Hom}}
\DeclareMathOperator{\Homsh}{\underline{\textup{Hom}}}
\DeclareMathOperator{\ISym}{\textup{Sym}^*}
\DeclareMathOperator{\Imm}{\textup{Im}}
\DeclareMathOperator{\Irr}{\textup{Irr}}
\DeclareMathOperator{\Iso}{\textup{Iso}}
\DeclareMathOperator{\Isosh}{\underline{\textup{Iso}}}
\DeclareMathOperator{\Ker}{\textup{Ker}}
\DeclareMathOperator{\LAdd}{\textup{LAdd}}
\DeclareMathOperator{\LAlg}{\textup{LAlg}}
\DeclareMathOperator{\LMon}{\textup{LMon}}
\DeclareMathOperator{\LPMon}{\textup{LPMon}}
\DeclareMathOperator{\LRings}{\textup{LRings}}
\DeclareMathOperator{\LSMon}{\textup{LSMon}}
\DeclareMathOperator{\Left}{\textup{L}}
\DeclareMathOperator{\Lex}{\textup{Lex}}
\DeclareMathOperator{\Loc}{\textup{Loc}}
\DeclareMathOperator{\M}{\textup{M}}
\DeclareMathOperator{\ML}{\textup{ML}}
\DeclareMathOperator{\MLex}{\textup{MLex}}
\DeclareMathOperator{\Map}{\textup{Map}}
\DeclareMathOperator{\Mod}{\textup{Mod}}
\DeclareMathOperator{\Mon}{\textup{Mon}}
\DeclareMathOperator{\Ob}{\textup{Ob}}
\DeclareMathOperator{\Obj}{\textup{Obj}}
\DeclareMathOperator{\PDiv}{\textup{PDiv}}
\DeclareMathOperator{\PGL}{\textup{PGL}}
\DeclareMathOperator{\PML}{\textup{PML}}
\DeclareMathOperator{\PMLex}{\textup{PMLex}}
\DeclareMathOperator{\PMon}{\textup{PMon}}
\DeclareMathOperator{\Pic}{\textup{Pic}}
\DeclareMathOperator{\Picsh}{\underline{\textup{Pic}}}
\DeclareMathOperator{\Pro}{\textup{Pro}}
\DeclareMathOperator{\Proj}{\textup{Proj}}
\DeclareMathOperator{\QAdd}{\textup{QAdd}}
\DeclareMathOperator{\QAlg}{\textup{QAlg}}
\DeclareMathOperator{\QCoh}{\textup{QCoh}}
\DeclareMathOperator{\QMon}{\textup{QMon}}
\DeclareMathOperator{\QPMon}{\textup{QPMon}}
\DeclareMathOperator{\QRings}{\textup{QRings}}
\DeclareMathOperator{\QSMon}{\textup{QSMon}}
\DeclareMathOperator{\R}{\textup{R}}
\DeclareMathOperator{\Rep}{\textup{Rep}}
\DeclareMathOperator{\Rings}{\textup{Rings}}
\DeclareMathOperator{\Riv}{\textup{Riv}}
\DeclareMathOperator{\SFibr}{\textup{SFibr}}
\DeclareMathOperator{\SMLex}{\textup{SMLex}}
\DeclareMathOperator{\SMex}{\textup{SMex}}
\DeclareMathOperator{\SMon}{\textup{SMon}}
\DeclareMathOperator{\SchI}{\textup{SchI}}
\DeclareMathOperator{\Sh}{\textup{Sh}}
\DeclareMathOperator{\Soc}{\textup{Soc}}
\DeclareMathOperator{\Spec}{\textup{Spec}}
\DeclareMathOperator{\Specsh}{\underline{\textup{Spec}}}
\DeclareMathOperator{\Stab}{\textup{Stab}}
\DeclareMathOperator{\Supp}{\textup{Supp}}
\DeclareMathOperator{\Sym}{\textup{Sym}}
\DeclareMathOperator{\TMod}{\textup{TMod}}
\DeclareMathOperator{\Top}{\textup{Top}}
\DeclareMathOperator{\Tor}{\textup{Tor}}
\DeclareMathOperator{\Vect}{\textup{Vect}}
\DeclareMathOperator{\alt}{\textup{ht}}
\DeclareMathOperator{\car}{\textup{char}}
\DeclareMathOperator{\codim}{\textup{codim}}
\DeclareMathOperator{\degtr}{\textup{degtr}}
\DeclareMathOperator{\depth}{\textup{depth}}
\DeclareMathOperator{\divis}{\textup{div}}
\DeclareMathOperator{\et}{\textup{et}}
\DeclareMathOperator{\ffpSch}{\textup{ffpSch}}
\DeclareMathOperator{\h}{\textup{h}}
\DeclareMathOperator{\ilim}{\displaystyle{\lim_{\longrightarrow}}}
\DeclareMathOperator{\ind}{\textup{ind}}
\DeclareMathOperator{\indim}{\textup{inj dim}}
\DeclareMathOperator{\lf}{\textup{LF}}
\DeclareMathOperator{\op}{\textup{op}}
\DeclareMathOperator{\ord}{\textup{ord}}
\DeclareMathOperator{\pd}{\textup{pd}}
\DeclareMathOperator{\plim}{\displaystyle{\lim_{\longleftarrow}}}
\DeclareMathOperator{\pr}{\textup{pr}}
\DeclareMathOperator{\pt}{\textup{pt}}
\DeclareMathOperator{\rk}{\textup{rk}}
\DeclareMathOperator{\tr}{\textup{tr}}
\DeclareMathOperator{\type}{\textup{r}}
\DeclareMathOperator*{\colim}{\textup{colim}}

%% file: packages_and_functions.tex
\global\long\def\A{\mathbb{A}}

\global\long\def\Ab{(\textup{Ab})}

\global\long\def\C{\mathbb{C}}

\global\long\def\Cat{(\textup{cat})}

\global\long\def\Di#1{\textup{D}(#1)}

\global\long\def\E{\mathcal{E}}

\global\long\def\F{\mathbb{F}}

\global\long\def\GCov{G\textup{-Cov}}

\global\long\def\Gcat{(\textup{Galois cat})}

\global\long\def\Gfsets#1{#1\textup{-fsets}}

\global\long\def\Gm{\mathbb{G}_{m}}

\global\long\def\GrCov#1{\textup{D}(#1)\textup{-Cov}}

\global\long\def\Grp{(\textup{Grps})}

\global\long\def\Gsets#1{(#1\textup{-sets})}

\global\long\def\HCov{H\textup{-Cov}}

\global\long\def\MCov{\textup{D}(M)\textup{-Cov}}

\global\long\def\MHilb{M\textup{-Hilb}}

\global\long\def\N{\mathbb{N}}

\global\long\def\PGor{\textup{PGor}}

\global\long\def\PGrp{(\textup{Profinite Grp})}

\global\long\def\PP{\mathbb{P}}

\global\long\def\Pj{\mathbb{P}}

\global\long\def\Q{\mathbb{Q}}

\global\long\def\RCov#1{#1\textup{-Cov}}

\global\long\def\RR{\mathbb{R}}

\global\long\def\Sch{\textup{Sch}}

\global\long\def\WW{\textup{W}}

\global\long\def\Z{\mathbb{Z}}

\global\long\def\acts{\curvearrowright}

\global\long\def\alA{\mathscr{A}}

\global\long\def\alB{\mathscr{B}}

\global\long\def\arr{\longrightarrow}

\global\long\def\arrdi#1{\xlongrightarrow{#1}}

\global\long\def\catC{\mathscr{C}}

\global\long\def\catD{\mathscr{D}}

\global\long\def\catF{\mathscr{F}}

\global\long\def\catG{\mathscr{G}}

\global\long\def\comma{,\ }

\global\long\def\covU{\mathcal{U}}

\global\long\def\covV{\mathcal{V}}

\global\long\def\covW{\mathcal{W}}

\global\long\def\duale#1{{#1}^{\vee}}

\global\long\def\fasc#1{\widetilde{#1}}

\global\long\def\fsets{(\textup{f-sets})}

\global\long\def\iL{r\mathscr{L}}

\global\long\def\id{\textup{id}}

\global\long\def\la{\langle}

\global\long\def\odi#1{\mathcal{O}_{#1}}

\global\long\def\ra{\rangle}

\global\long\def\set{(\textup{Sets})}

\global\long\def\sets{(\textup{Sets})}

\global\long\def\shA{\mathcal{A}}

\global\long\def\shB{\mathcal{B}}

\global\long\def\shC{\mathcal{C}}

\global\long\def\shD{\mathcal{D}}

\global\long\def\shE{\mathcal{E}}

\global\long\def\shF{\mathcal{F}}

\global\long\def\shG{\mathcal{G}}

\global\long\def\shH{\mathcal{H}}

\global\long\def\shI{\mathcal{I}}

\global\long\def\shJ{\mathcal{J}}

\global\long\def\shK{\mathcal{K}}

\global\long\def\shL{\mathcal{L}}

\global\long\def\shM{\mathcal{M}}

\global\long\def\shN{\mathcal{N}}

\global\long\def\shO{\mathcal{O}}

\global\long\def\shP{\mathcal{P}}

\global\long\def\shQ{\mathcal{Q}}

\global\long\def\shR{\mathcal{R}}

\global\long\def\shS{\mathcal{S}}

\global\long\def\shT{\mathcal{T}}

\global\long\def\shU{\mathcal{U}}

\global\long\def\shV{\mathcal{V}}

\global\long\def\shW{\mathcal{W}}

\global\long\def\shX{\mathcal{X}}

\global\long\def\shY{\mathcal{Y}}

\global\long\def\shZ{\mathcal{Z}}

\global\long\def\st{\ | \ }

\global\long\def\stA{\mathcal{A}}

\global\long\def\stB{\mathcal{B}}

\global\long\def\stC{\mathcal{C}}

\global\long\def\stD{\mathcal{D}}

\global\long\def\stE{\mathcal{E}}

\global\long\def\stF{\mathcal{F}}

\global\long\def\stG{\mathcal{G}}

\global\long\def\stH{\mathcal{H}}

\global\long\def\stI{\mathcal{I}}

\global\long\def\stJ{\mathcal{J}}

\global\long\def\stK{\mathcal{K}}

\global\long\def\stL{\mathcal{L}}

\global\long\def\stM{\mathcal{M}}

\global\long\def\stN{\mathcal{N}}

\global\long\def\stO{\mathcal{O}}

\global\long\def\stP{\mathcal{P}}

\global\long\def\stQ{\mathcal{Q}}

\global\long\def\stR{\mathcal{R}}

\global\long\def\stS{\mathcal{S}}

\global\long\def\stT{\mathcal{T}}

\global\long\def\stU{\mathcal{U}}

\global\long\def\stV{\mathcal{V}}

\global\long\def\stW{\mathcal{W}}

\global\long\def\stX{\mathcal{X}}

\global\long\def\stY{\mathcal{Y}}

\global\long\def\stZ{\mathcal{Z}}

\global\long\def\then{\ \Longrightarrow\ }

\global\long\def\L{\textup{L}}

\global\long\def\l{\textup{l}}

%% file: monoidal.galois.bbl
\providecommand{\bysame}{\leavevmode\hbox to3em{\hrulefill}\thinspace}
\providecommand{\MR}{\relax\ifhmode\unskip\space\fi MR }
\providecommand{\MRhref}[2]{%
  \href{http://www.ams.org/mathscinet-getitem?mr=#1}{#2}
}
\providecommand{\href}[2]{#2}
\begin{thebibliography}{Ton13b}

\bibitem[AOV08]{Abramovich2007}
Dan Abramovich, Martin Olsson, and Angelo Vistoli, \emph{{Tame stacks in
  positive characteristic}}, Annales de l'Institut Fourier \textbf{58} (2008),
  no.~4, 1057--1091.

\bibitem[DM82]{Deligne1982}
Pierre Deligne and James~S. Milne, \emph{{Tannakian Categories}}, Hodge cycles,
  motives, and Shimura varieties, Lecture Notes in Mathematics, vol. 900,
  Springer-Verlag, Berlin, 1982, pp.~101--228.

\bibitem[Gir71]{Giraud1971}
Jean Giraud, \emph{{Cohomologie non ab\'{e}lienne}}, Springer-Verlag, Berlin,
  1971.

\bibitem[Gro66]{EGAIV-3}
Alexander Grothendieck, \emph{{EGA4-3 - \'{E}tude locale des sch\'{e}mas et des
  morphismes de sch\'{e}mas (Troisi\'{e}m partie) - \'{E}l\'{e}ments de
  g\'{e}om\'{e}trie alg\'{e}brique (r\'{e}dig\'{e}s avec la collaboration de
  Jean Dieudonn\'{e})}}, 28 ed., Institut des Hautes \'{E}tudes Scientifiques.
  Publications Math\'{e}matiques, 1966.

\bibitem[Jan87]{Jantzen2003}
Jens~Carsten Jantzen, \emph{{Representations of algebraic groups}}, Pure and
  Applied Mathematics, vol. 131, Academic Press Inc., 1987.

\bibitem[Mat89]{Matsumura1989}
Hideyuki Matsumura, \emph{{Commutative ring theory. Translated from the
  Japanese by Miles Reid}}, Cambridge University Press, 1989.

\bibitem[MBL99]{Laumon1999}
Laurent Moret-Bailly and Gerard Laumon, \emph{{Champs alg\'{e}briques}}, first
  ed., Springer, 1999.

\bibitem[MM03]{Miller1903}
George~Abram Miller and Halcott~Cadwalader Moreno, \emph{{Non-Abelian Groups in
  Which Every Subgroup is Abelian}}, Transactions of the American Mathematical
  Society \textbf{4} (1903), no.~4, 398.

\bibitem[Sch13]{Schappi2011}
Daniel Sch\"{a}ppi, \emph{{The formal theory of Tannaka duality}},
  Ast\'{e}risque \textbf{357} (2013), viii+140.

\bibitem[Ton13a]{Tonini2011}
Fabio Tonini, \emph{{Stacks of Ramified Covers Under Diagonalizable Group
  Schemes}}, International Mathematics Research Notices (2013), 80.

\bibitem[Ton13b]{Tonini2013}
\bysame, \emph{{Stacks of ramified Galois covers}}, Ph.D. thesis (2013), 192.

\bibitem[Ton14]{Tonini2014}
\bysame, \emph{{Sheafification functors and Tannaka's reconstruction}},
  arXiv:1409.4073 (2014), 35.

\bibitem[Ton15]{Tonini2015}
\bysame, \emph{{Trace map and regularity of finite extensions of a DVR}},
  arXiv:1506.04264 (2015), 8.

\end{thebibliography}
